\pdfoutput=1
\documentclass{amsart}
\usepackage[mathscr]{eucal}
\usepackage{amssymb}
\usepackage[usenames,dvipsnames]{xcolor} 
\usepackage[normalem]{ulem}
\usepackage{wasysym}
\usepackage{amsthm}
\usepackage{bbold}
\usepackage{comment}
\usepackage{enumitem}
\usepackage{amsmath}
\usepackage{amssymb}
\usepackage{tikz-cd}
\usetikzlibrary{arrows}
\usepackage{stmaryrd} 
\usepackage{etoolbox} 
\usepackage{microtype}
\usepackage{adjustbox}
\usepackage[unicode]{hyperref} 

\definecolor{dark-red}{rgb}{0.5,0.15,0.15}
\definecolor{dark-blue}{rgb}{0.15,0.15,0.6}
\definecolor{dark-green}{rgb}{0.15,0.6,0.15}

\hypersetup{
    colorlinks, linkcolor=OliveGreen,
    citecolor=OliveGreen, urlcolor=OliveGreen
}

\usepackage[nameinlink,capitalise,noabbrev]{cleveref}

\numberwithin{equation}{section}
\usepackage[all]{xy}
\xyoption{line}
\usepackage{graphicx}
\usepackage{mathtools}
\usepackage{caption}

\newtheorem{Thm}[equation]{Theorem}
\newtheorem*{Thm*}{Theorem}
\newtheorem*{MainThm*}{Main Theorem}
\newtheorem{Prop}[equation]{Proposition}
\newtheorem{Lem}[equation]{Lemma}
\newtheorem{Cor}[equation]{Corollary}

\newtheorem*{Que*}{Question}

\theoremstyle{remark}
\newtheorem{Def}[equation]{Definition}
\newtheorem{Ter}[equation]{Terminology}
\newtheorem{Not}[equation]{Notation}
\newtheorem{Exa}[equation]{Example}

\newtheorem{Cons}[equation]{Construction}

\newtheorem{Rem}[equation]{Remark}

\tikzset{
    labelrotatebelow/.style={anchor=north, rotate=90, inner sep=1.0mm}
}
\tikzset{
    labelrotateabove/.style={anchor=south, rotate=90, inner sep=1.0mm}
}
\SelectTips{cm}{10}

\newcommand{\nc}{\newcommand}
\nc{\dmo}{\DeclareMathOperator}

\renewcommand{\emptyset}{\varnothing}

\nc{\Niko}[1]{{\color{Orange}#1}}
\nc{\Maxime}[1]{{\color{Green}#1}}
\nc{\Luca}[1]{{\color{Blue}#1}}

\usepackage{todonotes}

\nc{\overbar}[1]{\mkern 1.5mu\overline{\mkern-1.5mu#1\mkern-1.5mu}\mkern 1.5mu}

\nc{\kappaaux}{g}
\nc{\kappaCh}{{\kappaaux(\cat C_h)}}
\nc{\kappam}{{\kappaaux({\mathfrak m})}}
\nc{\kappaP}{\Gamma_{\cat P}\unit}
\nc{\kappaQ}{{\kappaaux(\cat Q)}}
\nc{\kappaCP}{{\kappaaux_{\cat C}(\cat P)}}
\nc{\kappaDP}{{\kappaaux_{\cat D}(\cat P)}}
\nc{\kappaCQ}{{\kappaaux_{\cat C}(\cat Q)}}
\nc{\kappaDQ}{{\kappaaux_{\cat D}(\cat Q)}}
\nc{\kappaphiB}{{\kappaaux(\phi(\cat B))}}
\nc{\kappaphiQ}{{\kappaaux(\varphi(\cat Q))}}

\dmo{\Sub}{Sub}
\dmo{\tw}{tw}
\dmo{\Proj}{Proj}
\dmo{\LMod}{LMod}
\dmo{\cell}{cell}
\nc{\SpEn}{\cat S_{E(n)}}
\nc{\SpEnf}{\cat S_n}
\nc{\Lcomp}{L^{\mathrm{com}}} 
\nc{\Ucomp}{U^{\mathrm{com}}}
\nc{\Loco}[1]{\Loc_{\otimes}\hspace{-0.3ex}\langle #1 \rangle}
\nc{\bbullet}{{\scriptscriptstyle\hspace{-1pt}\bullet}}
\nc{\bullett}{{\scriptscriptstyle\bullet}\hspace{-1pt}}
\nc{\LF}{L\hspace{-0.2ex}F}
\nc{\SpG}{\Sp^G}
\nc{\EG}{\bbE_G}
\nc{\DEG}{\Der(\EG)}
\nc{\DE}{\Der(\bbE)}
\nc{\Prst}{{\cat P}\mathrm{r^{st}}}
\nc{\Mack}[2]{\mathrm{Mack}_{#1}(#2)}
\nc{\SC}{S\cat C}
\dmo{\fin}{{fin}}
\dmo{\DM}{DM}
\dmo{\fp}{fp}
\nc{\DMQ}{\DM_Q}
\dmo{\DerKal}{DMack}
\dmo{\Der}{D}
\dmo{\DMot}{DMot}
\dmo{\rmH}{H}
\dmo{\piu}{\underline{\pi}}
\dmo{\Sphere}{\mathbb{S}}
\dmo{\Alg}{Alg}
\dmo{\CAlg}{CAlg}
\nc{\HA}{{\rmH \hspace{-0.2em}\bbA}}
\nc{\HZ}{{\rmH \hspace{-0.2em}\bbZ}}
\nc{\HZbar}{{\rmH \hspace{-0.2em}\underline{\bbZ}}}
\nc{\Fp}{{\bbF_{\hspace{-0.1em}p}}}
\nc{\HFp}{{\rmH \hspace{-0.15em}\bbF_{\hspace{-0.1em}p}}}
\nc{\DHZG}{\Der(\HZ_G)}
\nc{\DHZH}{\Der(\HZ_H)}
\nc{\DHZK}{\Der(\HZ_K)}
\nc{\DHZGN}{\Der(\HZ_{G/N})}
\nc{\DHZGG}{\Der(\HZ_{G/G})}
\nc{\DHZCp}{\Der(\HZ_{C_p})}
\nc{\DHZGprime}{\Der(\HZ_{G'})}
\nc{\DHZ}{\Der(\HZ)}
\nc{\mathfrakp}{\mathfrak{p}}
\nc{\mathfrakq}{\mathfrak{q}}
\nc{\mathfrakS}{\mathfrak{S}}
\nc{\mathfrakT}{\mathfrak{T}}
\nc{\Z}{\mathbb{Z}}
\nc{\SSG}{\text{sSet}_*^G}
\nc{\sSet}{\text{sSet}}

\dmo{\csupp}{csupp}
\dmo{\Con}{Conj}
\dmo{\Id}{Id}
\dmo{\Loc}{Loc}
\dmo{\rmK}{\textrm{\rm K}}
\dmo{\Spc}{Spc}
\dmo{\thick}{thick}
\nc{\thickt}[1]{\thick_\otimes\langle #1 \rangle}
\nc{\loct}[1]{\loc_\otimes\langle #1 \rangle}
\dmo{\cone}{cone}
\dmo{\End}{End}
\dmo{\Derperf}{D_{perf}}
\dmo{\Mor}{Mor}
\dmo{\Hom}{Hom}
\dmo{\id}{id}
\dmo{\incl}{incl}
\dmo{\Img}{Im}
\dmo{\im}{im}
\dmo{\Ker}{Ker}
\dmo{\ind}{ind}
\dmo{\Ind}{Ind}
\dmo{\CoInd}{coind}
\dmo{\res}{res}
\dmo{\infl}{infl}
\dmo{\Derqc}{D_{qc}}
\dmo{\triv}{triv}
\dmo{\sep}{sep}
\dmo{\Tel}{Tel} 
\dmo{\grMod}{grMod}%
\dmo{\Mod}{Mod}%
\dmo{\opname}{op}
\dmo{\SH}{SH}
\dmo{\smallb}{b}
\dmo{\Spec}{Spec}
\dmo{\supp}{supp}
\dmo{\Supp}{Supp}
\dmo{\cosupp}{cosupp}
\dmo{\Cosupp}{Cosupp}
\nc{\SHc}{{\SH^c}}
\nc{\SHp}{{\SH_{(p)}}}
\nc{\SHcp}{{\SH^c_{(p)}}}
\nc{\SHG}{\SH(G)}
\nc{\SHGp}{\SH(G)_{(p)}}
\nc{\SHGc}{\SHG^c}
\nc{\SHGcp}{\SHG^c_{(p)}}
\nc{\quadtext}[1]{\quad\textrm{#1}\quad}
\nc{\qquadtext}[1]{\qquad\textrm{#1}\qquad}
\nc{\adj}{\dashv}
\nc{\adjto}{\rightleftarrows}
\nc{\bbL}{\mathbb{L}}
\nc{\bbA}{\mathbb{A}}
\nc{\bbE}{\mathbb{E}}
\nc{\bbN}{\mathbb{N}}
\nc{\bbQ}{\mathbb{Q}}
\nc{\bbZ}{\mathbb{Z}}
\nc{\bbF}{\mathbb{F}}
\nc{\cat}[1]{\mathscr{#1}}
\nc{\ie}{{\sl i.e.}, }
\nc{\into}{\mathop{\rightarrowtail}}
\nc{\inv}{^{-1}}
\nc{\isoto}{\mathop{\overset{\sim}\to}}
\nc{\isotoo}{\mathop{\overset{\sim}\too}}
\nc{\onto}{\mathop{\twoheadrightarrow}}
\nc{\too}{\mathop{\longrightarrow}\limits}
\nc{\mapstoo}{\longmapsto}
\nc{\adh}[1]{\overline{#1}}
\nc{\adhpt}[1]{\adh{\{#1\}}}
\nc{\aka}{{a.\,k.\,a.}\ }
\nc{\calF}{\mathcal{F}}
\nc{\eg}{{\sl e.\,g.}}
\nc{\Homcat}[1]{\Hom_{\cat #1}}
\nc{\hook}{\hookrightarrow}
\nc{\ideal}[1]{\langle #1\rangle}
\nc{\ihom}{{\underline{\hom}}}
\nc{\iHom}{\underline{\mathrm{Hom}}}
\nc{\Mid}{\,\big|\,}
\nc{\MMod}{\,\text{-}\Mod}%
\nc{\GrMMod}{\,\text{-}\grMod}%
\nc{\op}{^{\opname}}
\nc{\oto}[1]{\overset{#1}\to}
\nc{\otoo}[1]{\overset{#1}{\,\too\,}}
\nc{\sminus}{\!\smallsetminus\!}
\nc{\poplus}[1]{^{\oplus #1}}%
\nc{\potimes}[1]{^{\otimes #1}}
\nc{\sbull}{{\scriptscriptstyle\bullet}}
\nc{\SET}[2]{\big\{\,#1\Mid#2\,\big\}}
\nc{\SpcK}{\Spc(\cat K)}
\nc{\then}{\Rightarrow}
\nc{\unit}{\mathbb{1}}
\nc{\xra}{\xrightarrow}
\nc{\phigeom}[1]{\widetilde{\Phi}^{#1}}
\dmo{\Oname}{O}
\dmo{\proper}{proper}
\dmo{\lenormal}{\unlhd}
\dmo{\lnormal}{\lhd}
\nc{\normal}{\trianglelefteq}
\nc{\Op}{\Oname^p}
\nc{\Oq}{\Oname^q}
\dmo{\Sp}{Sp}
\dmo{\Ho}{Ho}
\dmo{\Fin}{Fin}
\dmo{\add}{add}
\dmo{\Fun}{Fun}
\dmo{\Ext}{Ext}

\dmo{\CMon}{CMon}
\dmo{\BB}{\cat B}
\dmo{\CC}{\cat C}
\dmo{\DD}{\cat D}
\dmo{\MM}{\cat M}
\dmo{\NN}{\cat N}
\dmo{\OO}{\mathcal{O}}
\dmo{\Map}{Map}
\dmo{\Span}{Span}
\dmo{\N}{N}
\dmo{\Cat}{Cat}
\dmo{\colim}{colim}
\dmo{\hocolim}{hocolim}
\dmo{\Ch}{Ch}
\dmo{\A}{\mathbb{A}^{eff}}
\nc{\AGeff}{\mathbb{A}_G^{\mathrm{eff}}}
\nc{\BGeff}{\mathcal{B}_G^{\mathrm{eff}}}
\nc{\BG}{{\mathcal{B}_G}}
\nc{\NBGeff}{{\N}{\BGeff}}
\dmo{\Ab}{Ab}
\dmo{\Set}{Set}
\dmo{\ev}{ev}
\dmo{\Spcl}{Spcl}
\nc{\Funadd}{\Fun_{\add}}
\dmo{\proj}{proj}
\dmo{\cof}{cof}

\nc{\StModfin}{\mathrm{StMod}^{\mathrm{fin}}}
\nc{\PrL}{\mathrm{Pr}^{\mathrm{L}}_{\mathrm{st}}}

\dmo{\Coideal}{Coideal}
\dmo{\gen}{gen}
\dmo{\Coloc}{Coloc}
\nc{\Coloco}[1]{\Coloc^{\iHom}\hspace{-0.3ex}\langle #1 \rangle}

\dmo{\dual}{dual}

\nc{\LOCO}{\mathcal{L}\mathrm{oc}_{\otimes}}
\nc{\COLOCO}{\mathcal{C}\mathrm{oloc}^{\iHom}}

\nc{\Perff}[1]{\mathrm{Perf}_{#1}}
\nc{\Modd}[1]{\mathrm{Mod}_{#1}}
\dmo{\Perf}{Perf}
\dmo{\tel}{tel}
\nc{\Lococat}[2]{\Loc^{#1}_{\otimes}\hspace{-0.3ex}\langle #2 \rangle}
\dmo{\rk}{rk}
\dmo{\StMod}{StMod}
\dmo{\stmod}{stmod}

\nc{\CSep}{\mathcal{C}\mathcal{S}\mathrm{ep}}
\nc{\CFin}{\mathcal{F}\mathrm{in}}
\nc{\Ccpl}[1]{\CC^{#1\text{-}\mathrm{cpl}}}
\nc{\Ctors}[1]{\CC^{#1\text{-}\mathrm{tors}}}
\nc{\Cloc}[1]{\CC^{#1\text{-}\mathrm{loc}}}
\nc{\fib}[1]{\mathrm{fib}}
\nc{\tors}{\mathrm{tors}}
\nc{\cpl}{\mathrm{cpl}}
\nc{\loc}{\mathrm{loc}}

\nc{\doublefaktor}[3]{%
    {\textstyle #1}
    \mkern-4mu\scalebox{1.5}{$\diagdown$}\mkern-5mu^{\textstyle #2}%
    \mkern-4mu\scalebox{1.5}{$\diagup$}\mkern-5mu{\textstyle #3} }

\nc{\QW}{W^{\text{Q}}}

\nc{\mT}{\kern-0.5em\mod\kern-0.1em\text{-}\cat{T}^c}
\nc{\mTc}{\kern-0.5em\mod\kern-0.1em\text{-}\cat{T}^c}
\nc{\MTc}{\Mod\kern-0.1em\text{-}\cat{T}^c}
\nc{\MT}{\Mod\kern-0.1em\text{-}\cat{T}}
\newcounter{enum-resume-hack}
\Crefname{Thm}{Theorem}{Theorems}
\Crefname{Prop}{Proposition}{Propositions}
\Crefname{Lem}{Lemma}{Lemmas}
\Crefname{thmx}{Theorem}{Theorems}

\begin{document}
\setlength{\parindent}{0cm}
\setlength{\parskip}{0.8ex}
\title{Separable commutative algebras in equivariant homotopy theory}

\author[Naumann]{Niko Naumann}
\author[Pol]{Luca Pol}
\author[Ramzi]{Maxime Ramzi}
\date{\today}

\address{Niko Naumann, Fakult{\"a}t f{\"u}r Mathematik, Universit{\"a}t Regensburg, Universit{\"a}tsstraße 31, 93053 Regensburg, Germany}
\email{Niko.Naumann@mathematik.uni-regensburg.de}
\urladdr{https://homepages.uni-regensburg.de/$\sim$nan25776/}

\address{Luca Pol, Fakult{\"a}t f{\"u}r Mathematik, Universit{\"a}t Regensburg, Universit{\"a}tsstraße 31, 93053 Regensburg, Germany}
\email{luca.pol@mathematik.uni-regensburg.de}
\urladdr{https://sites.google.com/view/lucapol/}

\address{Maxime Ramzi,
FachBereich Mathematik und Informatik, Universit{\"a}t M{\"u}nster, Einsteinstraße 62 Germany}
\email{mramzi@uni-muenster.de}
\urladdr{https://sites.google.com/view/maxime-ramzi-en/home}

\begin{abstract}
Given a finite group $G$ and a commutative ring $G$-spectrum $R$, we study the separable commutative algebras in the category of compact $R$-modules. We isolate three conditions on the geometric fixed points of $R$ which ensure that every separable commutative algebra is standard, i.e. arises from a finite $G$-set. 
In particular we show that all separable commutative algebras in the categories of compact objects in $G$-spectra and in derived $G$-Mackey functors are standard provided that $G$ is a $p$-group. In these categories we also show that for a general finite group $G$, not all separable commutative algebras are standard. We finally discuss how the classification of separable commutative algebras in compact $G$-spectra varies if we require the existence of multiplicative norms. We show that if $G$ is solvable, then any separable commutative algebra therein that is normed is automatically standard. However, if $G$ is not solvable, we provide examples of separable commutative algebras that are normed but not standard.  
\end{abstract}


\maketitle
\setcounter{tocdepth}{1}
\tableofcontents

\section{Introduction}

The study of separable commutative algebras in the setting of tensor-triangular geometry was initiated by Balmer in \cite{Balmer2011}. Separable commutative algebras are useful in this context for several reasons.
Firstly, they allow to form categories of modules internally to the world of triangulated categories, without the necessity of assuming that the triangulated category admits a model. Building on this, Balmer showed that restriction functors in (modular) representation theory can be understood as an extension-of-scalars functors \cite{Balmer2015}, and proved a generalization of the Neeman-Thomason Localization Theorem \cite{Balmer2016-etale}.
Secondly, separable commutative algebras allow us to formulate a going-up theorem in tt-geometry \cite{Balmer2016}, which is a useful tool for calculating the Balmer spectrum of a triangulated category, and powerful descent results in tt-geometry~\cite{Balmer2016},\cite{descent} and \cite{BCHNP}.
Finally, separable commutative algebras are tightly related to the \'etale topology of the given tt-category. A result of Neeman \cite{Neeman2018} informally says that the only separable commutative algebras in the perfect derived category of a Noetherian scheme are those arising from finite \'etale morphisms. This connection with the \'etale topology was strengthened even further by the first two authors in \cite{NaumannPol} where they showed that separable commutative algebras recover finite \'etale algebras in various contexts where a notion of \'etaleness already exists. 

The study and classification of separable commutative algebras is more complicated when some equivariance comes into play. Given a finite group $G$ and a separably closed field $k$, Balmer showed that the only separable commutative algebras in the category of $kG$-modules are those of the form $kX$ for some finite $G$-set $X$. We will informally refer to this result as saying that in the category of $kG$-modules all separable algebra are \emph{standard}. Balmer went one step further and asked if the same holds in the stable module category \cite[Question 4.7]{Balmer2015}:
\begin{Que*}[Balmer]
Let $k$ be a separably closed field and $A\in \stmod(kG)$ be a separable commutative algebra. Is there a finite $G$-set $X$ such that $A\simeq  kX$ in $\stmod(kG)$?
\end{Que*}
Informally this problem asks whether the \'etale topology which appears in modular representation theory is richer than what is produced by subgroups. A positive answer to this question was given by Balmer-Carlson \cite{BC2018} in the case of cyclic $p$-groups and by the first two authors \cite{NaumannPol} for $p$-groups of $p$-rank one. This question is completely open outside of these cases.  

In this paper we aim to address the analogue of Balmer's question in the setting of equivariant stable homotopy theory. Fix a finite group $G$ and a commutative ring $G$-spectrum $R \in \CAlg(\Sp_G)$, i.e. a commutative algebra in genuine $G$-spectra. Given a finite $G$-set $X$, we can form the commutative algebra $\mathbb{D}(X_+) \in \CAlg(\Sp_G)$ by taking the dual of the suspension spectrum of $X_+$; this is separable by work of Balmer-Dell'Ambrogio-Sanders~\cite{Balmer_Ambrogio_Sanders}. Base-changing to our $R$ we obtain a functor 
\[
R^{(-)}\colon \Fin\op_G \to \CAlg^{\sep}(\Mod_{\Sp_G}(R)^\omega)  
\]
into the category of separable commutative algebras in compact $R$-modules. We will denote the restriction to groupoids of this functor as
\[
R^{(-)}\colon \CFin\op_G \to \CSep(\Mod_{\Sp_G}(R)^\omega).
\]
In this paper we isolate three conditions that ensure that this functor is an equivalence, showing in particular that every separable algebra is standard. 
Recall that for all subgroups $K \subseteq G$, the $K$-geometric fixed points $\Phi^K(R)$ carries an action of the Weyl group $W_G(K)$ so that  
\[
\Phi^K(R) \in \Fun(BW_G(K), \CAlg).
\]
We then consider the following conditions for $(R,K)$:
\begin{itemize}
    \item[(1)] For all $1 \not = U \subseteq W_G(K)$, the commutative algebras $(\Phi^K(R))^{hU}$ and $ (\Phi^K(R))^{tU}$ are indecomposable. 
    \item[(2)] If $\Phi^K(R)$ is a retract in $\Mod(\Phi^K(R))^{t U}$ of the coinduced algebra $(\Phi^K(R))^{U/V}$ for some $V \subseteq U \subseteq W_G(K)$, then $U=V$, see \Cref{def-retraction condition}. 
    \item[(3)] $\Phi^K(R)$ is separably closed, see \Cref{def-sep}.
    \item[(4)] $\pi_0^K(R)$ is indecomposable and $\pi_0(\Phi^K R)$ is nonzero.
\end{itemize}
We will show that the above conditions are satisfied for all $K$  if $R=\mathbb{S}_G$ is the sphere $G$-spectrum, and also if $R=\infl_e^G \bbZ$, provided that $G$ is a $p$-group. The latter case has gained increasing interest for its connection to the category of derived $G$-Mackey functors of Kaledin, see \cite{PSW}. 
Our main result then states:

\begin{Thm*}[\ref{main-theorem}, \ref{prop-fully-faulful}]
    Suppose that $G$ and $R$ satisfy conditions (1), (2) and (3) for all $K\subseteq G$. Then 
    \[
    R^{(-)}\colon \CFin\op_G \stackrel{\simeq}{\longrightarrow} \CSep(\Mod_{\Sp_G}(R)^\omega)
    \]   
    is an equivalence, and so any separable commutative algebra is standard. If moreover (4) holds for all $K \subseteq G$, then we have an equivalence 
    \[
    R^{(-)}\colon \Fin\op_G \stackrel{\simeq}{\longrightarrow}  \CAlg^{\sep}(\Mod_{\Sp_G}(R)^\omega).
\]
\end{Thm*}

It follows that if $G$ is a $p$-group, then all separable commutative algebras are standard, both in compact $G$-spectra and in compact derived $G$-Mackey functors. In fact, for these categories, we can show that we have an equivalence before passing to groupoids. The techniques developed in this paper also allow us to produce some counterexamples. Indeed, we show in \Cref{sec-nonall} that if $G=C_6$, then there are separable commutative algebras in the above two categories that are not standard.  

As an immediate application of our work, we obtain the following computations of the Galois group in the sense of \cite{Mathew2016}. 

\begin{Thm*}[\ref{cor-galois}, \ref{cor-galois-cpq}]
   Let $p\neq q$ be prime numbers. 
   \begin{itemize}
       \item[(a)] If $G$ is a $p$-group, then the Galois groups of $\Sp_G$ is trivial. 
       \item[(b)] If $G$ is cyclic of order $pq$, then the Galois group of $\Sp_G$ is $\widehat{\Z}$.
   \end{itemize}
\end{Thm*}

In \cref{section-normed} we discuss how the classification of separable commutative algebras in compact $G$-spectra varies if we require the existence of multiplicative norms. We show that if $G$ is solvable, then any separable commutative algebra therein which is normed is automatically standard. However, if $G$ is not solvable, we provide examples of separable commutative algebras that are normed but not standard.

\subsection*{Acknowledgements} We thank Bastiaan Cnossen, John Greenlees, Achim Krause and Phil P\"{u}tzst\"{u}ck for useful conversations. We also thank the anonymous referee for their careful reading, helpful comments and for suggesting to consider the classification of normed separable algebras.

 NN was supported by the SFB 1085 Higher Invariants in Regensburg. LP is grateful to Max Planck Institute for Mathematics in Bonn for its hospitality and financial support. LP was supported by the European Research Council (ERC) under Horizon Europe (grant No.~101042990) and by the SFB 1085 Higher Invariants in Regensburg. MR was supported by the Danish National Research Foundation through the Copenhagen Centre for Geometry and Topology (DNRF151) by the Deutsche Forschungsgemeinschaft (DFG, German Research Foundation) - Project-ID 427320536 - SFB 1442, as well as under Germany’s Excellence Strategy EXC 2044 390685587, Mathematics Münster: Dynamics-Geometry-Structure. 

\subsection*{Notations}
\begin{enumerate}
\item We denote by $\Cat_\infty$ the $\infty$-category of $\infty$-categories.    \item We denote by $\CAlg$ the $\infty$-category of $\mathbb{E}_\infty$-algebras in spectra and by $\CAlg^{\heartsuit}$ the category of discrete commutative rings. 
    \item We denote by $\Cat_\infty^{\mathrm{st}}$ the $\infty$-category of essentially small stable $\infty$-category and exact functors. 
    \item We denote by $\CAlg(\Cat_\infty^{\mathrm{st}})$ the $\infty$-category of essentially small, symmetric monoidal stable $\infty$-categories whose tensor product is exact in each variable, and symmetric monoidal exact functors. 
    \item We denote by $\Cat_\infty^{\mathrm{perf}}$ the $\infty$-category of essentially small, idempotent complete stable $\infty$-categories and exact functors. 
    \item We denote by $\CAlg(\Cat_\infty^{\mathrm{perf}})$ the $\infty$-category of $2$-rings and symmetric monoidal exact functors. Recall that a $2$-ring is an essentially small, idempotent complete, symmetric monoidal and stable $\infty$-category whose tensor product is exact in each variable. 
    \item We denote by $\CAlg(\PrL)$ the $\infty$-category of presentably symmetric monoidal stable $\infty$-categories and symmetric monoidal left adjoints. 
    \item We denote by $\loc_\otimes$ (resp. $\thick_\otimes$) the localizing (resp. thick) tensor ideal generated by a given class of objects.
    \item For a finite group $G$, we let $\Sp_G$ denote the $\infty$-category of genuine $G$-spectra. 
    \item Given $R \in \CAlg(\Sp_G)$, we write $\Mod_{\Sp_G}(R)$ for the $\infty$-category of $R$-modules in $\Sp_G$ and $\Mod_{\Sp_G}(R)^\omega$ for its subcategory of compact objects. In the special case $G=e$, we know that the compact objects are equal to the perfect $R$-modules and so we will write $\Perf(R)$ instead of 
    $\Mod_{\Sp}(R)^\omega$.
\end{enumerate}

\section{Preliminaries}

In this section we fix some notation and recall a few key results that we will use throughout the paper. 

 Consider a symmetric monoidal stable $\infty$-category $\cat C$ and let $A \in \CAlg(\cat C)$ be a commutative algebra object. Recall from~\cite{Balmer2011} that $A$ is \emph{separable} if the multiplication map $A \otimes A \to A$ admits an $A$-bimodule section. We will write $\CAlg^{\sep}(\cat C)$ for the full subcategory of commutative algebra objects $\CAlg(\cat C)$ spanned by the separable algebras. It is immediate from the definition that any separable commutative algebra in $\cat C$ induces a separable commutative algebra object in the homotopy category $\Ho(\cat C)$. The following result gives us a converse.

\begin{Thm}[cf. {\cite[Theorem B]{Ramzi2023}}]\label{thm-separable-HO}
    Let $\cat C$ be a symmetric monoidal stable $\infty$-category, and let $A$ be an algebra in $\Ho(\cat C)$ which is separable and commutative. Then $A$ admits an essentially unique lift $\widetilde{A}\in \CAlg(\cat C)$ (and $\tilde A$ is separable). Furthermore, for any $R \in \CAlg(\cat C)$, the canonical map 
    \[
    \Map_{\CAlg(\cat C)}(\widetilde{A}, R) \xrightarrow{\sim} \Map_{\CAlg(\Ho(\cat C))}(A,R)
    \]
    is an equivalence.
\end{Thm}

A remarkable consequence of the previous result is that for any symmetric monoidal stable $\infty$-category $\cat C$, the $\infty$-category $\CAlg^{\sep}(\cat C)$ is in fact a $1$-category. 

\begin{Not}\label{not-CAlgsep}
 Let $\CSep(-)$ denote the composite functor $\CAlg^{\sep}(-)^\simeq$ which takes the maximal groupoid of the subcategory of separable commutative algebra.    
\end{Not}

We recall the following result from ~\cite[Corollary 3.24]{Ramzi2023}, which in a slightly different form also appears in~\cite[Proposition 6.3]{NaumannPol}.

\begin{Prop}\label{prop-sep-preserves-limits}
    The functor 
    \[
    \CAlg^{\sep}(-)\colon \CAlg(\Cat_{\infty}^{\mathrm{st}}) \to \Cat_\infty
    \]
    preserves limits. Consequently, $\CSep(-)$ also preserves limits.
\end{Prop}

\section{Homotopy $G$-fixed points and separable algebras}

In this section we will work with $\infty$-categories which are equipped with an action of a finite group. We start by recalling the definition of the homotopy $G$-fixed points of an $\infty$-category with $G$-action, and list various equivariant preliminaries regarding these categories. We then explain how, under the additional assumption that the $\infty$-category is suitably symmetric monoidal, the cotensor functor lifts to a functor taking values in a homotopy $G$-fixed points category. 

\begin{Def}
    Given a group $G$ and $\cat C \in \Fun(BG, \Cat_\infty)$, we can form 
    \[
    \cat C^{hG} := \lim_{BG}\cat C \in\Cat_\infty.
    \]
\end{Def}

\begin{Exa}
    If $G$ acts trivially on $\cat C$, then $\cat C^{hG} \simeq \Fun(BG,\cat C)$.
\end{Exa}

\begin{Rem}\label{rem:endring}
    If $\cat C\in\Fun(BG,\Pr^L_{\mathrm st})$, then the endomorphism ring of the unit object in $\cat C^{hG}$ is given by 
    \[
    \CAlg\ni\Map_{\cat C^{hG}}(\unit,\unit)\simeq \lim_{BG} \Map_{\cat C}(\unit,\unit)\simeq \Map_{\cat C}(\unit,\unit)^{hG}.
    \]
\end{Rem}

We record the following result from \cite[Proposition 6.7]{Glueingpaper} which in the case $G$ acts trivially on $\cat C$, states the familiar fact that the restriction functor admits a biadjoint (induction and coinduction).

\begin{Lem}\label{twistedinductioncoinduction}
Let $G$ be a finite group and $\cat C\in \Fun(BG, \Cat_\infty)$ whose underlying $\infty$-category is additive. For any subgroup $H \subseteq G$, the symmetric monoidal restriction functor
\[ 
\res_H^G\colon \cat C^{hG}\to \cat C^{hH}
\]
admits a left and a right adjoint, denoted respectively by $G \circledast_H (-)$ and $(-)^{\circledast_H G}$, and these two functors are equivalent. Furthermore, these adjunctions preserve dualizable objects.
\end{Lem}

\begin{Rem}\label{rem-proj-formulas}
    We observe that we always have canonical projection maps
    \[
    (s)^{\circledast_H G}\otimes t \to (s \otimes \res^G_H t)^{\circledast_H G}
    \]
  for all $t$ and $s$. Moreover, it is formal that these projection maps are equivalences if $t$ is dualizable, see for instance~\cite[Lemma 3.8]{Glueingpaper}.
\end{Rem}

\begin{Rem}\label{rem-Mackey-dec}
  We have a Mackey formula, namely for $H,K\subseteq G$, there is an equivalence 
  \[
  \res_H^G (G\circledast_K -) \simeq \bigoplus_{g\in H\backslash G/K} H\circledast_{H\cap K^g} \res_{H\cap K^g}^{K} (-)
  \]
  where $K^g=gKg^{-1}$, and the restriction from $K$ to $H\cap K^g$ is along the conjugation map $H\cap K^g\subset K^g\cong K$. This is a consequence of \cite[Lemma 4.9]{CLS2023} with $S=T=BG$ and $\mathcal C$  the $BG$-category given by
\[ (*) \,\,\,  BG^{\mathrm{op}}\simeq BG\xrightarrow{A}\CAlg\xrightarrow{\Mod(-)} \Cat_\infty. \]
The reference shows that this is $BG$-cocomplete so in particular,  the Beck-Chevalley transformation associated to the pullback in $\cat S_{/BG}\simeq \mathrm{PSh}(BG)$
  \[
  \begin{tikzcd}
      \coprod_{g\in H\backslash G/K}G/(H \cap K^g) \arrow[r]\arrow[d] & BH \arrow[d]\\
      BK \arrow[r] & BG
  \end{tikzcd}
  \]
  is an equivalence. One also checks that the limit-preserving extension of $(*)$ to $\cat S_{/BG}$ sends $BH\to BG$ to $\Mod(A)^{hH}$.\\
  For example, when $K=e$ is the trivial subgroup, we find $$\res_H^G (G\circledast_e -) \simeq \bigoplus_{H \backslash G} (H\circledast_e -).$$
\end{Rem}

In the following construction we explain how cotensoring defines a symmetric monoidal functor of the form $\Fin\op\to \CAlg(\cat C)$, $X\mapsto \unit^X:= \prod_X \unit$.
This construction will be natural in symmetric monoidal $\infty$-categories $\cat C$ with finite products in which the tensor product preserves finite products in each variable, and finite product preserving symmetric monoidal functors between them. 
\begin{Cons}\label{con-contensor-functor}
Let $\Cat_{\infty}^\Pi$ be the $\infty$-category of $\infty$-categories with finite products and finite product-preserving functors between them. It acquires a Lurie tensor product $\otimes$, following a dual construction to \cite[Section 4.8.1]{HA}. That is, $\cat D\otimes \cat E$ is such that finite product-preserving functors out of it into $\cat F$ are the same as functors $\cat D\times \cat E\to \cat F$ that preserve finite products in each variable separately.\\ 
As in \cite[Remark 4.8.1.9]{HA}, commutative algebras in $(\Cat^\Pi_\infty,\otimes)$ are exactly symmetric monoidal $\infty$-categories with finite products where the tensor product commutes with finite products in each variable. 
By design, the internal hom in $(\Cat^\Pi_\infty,\otimes)$ is given by the category of finite product-preserving functors, and so $\Fin\op$ is the unit for this symmetric monoidal structure since it is freely generated by the point under finite products\footnote{Note that the finite products in $\Fin\op$ are finite coproducts in $\Fin$.}.\\
In particular, $\Fin\op$ is canonically and uniquely a commutative algebra in this category, and since the cocartesian symmetric monoidal structure\footnote{Given by cartesian products in $\Fin$.} is such a commutative algebra structure, it must be the relevant one. 
 Thus, there is a symmetric monoidal, product-preserving functor $\Fin\op\to \cat C$, natural in $\cat C\in\CAlg((\Cat^\Pi_\infty,\otimes))$, where $\Fin\op$ is equipped with the cocartesian symmetric monoidal structure. Since the forgetful functor induces a symmetric monoidal equivalence $\CAlg(\Fin\op)\simeq \Fin\op$, this further induces a natural symmetric monoidal functor 
 \[ \unit^\bullet \colon \Fin\op\to\CAlg(\cat C).\]
 Finally, we note that for any finite set $X$, the commutative algebra $\unit^X$  is a finite product of the unit object, so separable. Thus, the above functor lands naturally in $\CAlg^{\sep}(\cat C)$. If $\cat C$ is semi-additive, which it will be in the rest of the paper, it is also dualizable, and thus the functor lands in $\CAlg^{\sep}(\cat C^{\dual}).$\\ Informally, pullback along the diagonal of a finite set $X$ is a separable, commutative algebra structure
 \[
 \unit_\cat C^X\otimes\unit_\cat C^X\simeq\unit_\cat C^{X\times X}\to \unit_{\cat C}^X, 
 \]
natural in $\cat C\in\CAlg((\Cat_\infty^\Pi,\otimes))$, and satisfying $\unit_\cat C^{X\coprod Y}\simeq\unit_\cat C^X \times\unit_\cat C^Y$.
\end{Cons}

\begin{Cons}\label{con-cotensor-G-equiv}
Consider $\cat C \in \Fun(BG,\CAlg(\Cat_\infty^\Pi))$ which is in addition semi-additive. The naturality of \Cref{con-contensor-functor} endows the functor 
\[ \unit^\bullet \colon \Fin\op\to\CAlg^{\sep}(\cat C^{\dual})\]
with a $G$-equivariant structure with respect to the trivial $G$-action on the source, and the canonical action coming from $\cat C$ on the target. By taking $G$-homotopy fixed points we obtain the symmetric monoidal and finite product-preserving functor  
\[
(\unit^\bullet)^{hG} \colon\Fin_G\op\to\CAlg^{\sep}(\cat C^{\dual})^{hG}.
\]
By construction, for any subgroup $H\subseteq  G$ there is a natural comparison square: 
\[\begin{tikzcd}
	{\Fin_G\op} & {\CAlg^{\sep}(\cat C^{\dual})^{hG}} \\
	{\Fin_H\op} & {\CAlg^{\sep}(\cat C^{\dual})^{hH}}
	\arrow["{(\unit^\bullet)^{hG}}",from=1-1, to=1-2]
	\arrow["\res^G_H"',from=1-1, to=2-1]
	\arrow["\res_H^G", from=1-2, to=2-2]
	\arrow["{(\unit^\bullet)^{hH}}", from=2-1, to=2-2]
\end{tikzcd}\]
which is vertically right adjointable by \cite[Lemma 6.9]{Glueingpaper}. 
\end{Cons}

\section{Stable module categories and the Tate construction}
In this section we work with a commutative algebra $A\in\CAlg$ equipped with an action of a finite group $G$ and define three main categories 
\[
\Perf(A)^{hG} \quad \StModfin(AG) \quad \Perf(A)^{tG}
\]
which we will use throughout the paper. We start off with the first category.
\begin{Def}\label{def-hG-modA}
Let $G$ be a group and consider $A\in \Fun(BG, \CAlg)$. The category of $A$-modules inherits a $G$-action, yielding a functor
\[
\Mod(A)\colon BG \to \PrL.
\]
We can then form $\Mod(A)^{hG}$
which we think of as the category of $A$-modules endowed with a semi-linear $G$-action. 
We note that the functor $\Mod(A)$ naturally lifts to $\CAlg(\PrL)$ and so canonically $\Mod(A)^{hG}\in \CAlg(\PrL)$.
To connect this with the set up of the previous section, recall that the forgetful functor $\PrL\to \Cat_\infty$ preserves and creates limits, see proof of \cite[Lemma 2.6]{NaumannPol}.
\end{Def}

\begin{Rem}
    Using that taking dualizable objects commutes with limits (see \cite[Proposition 4.6.1.11]{HA}), we see that
    \[ \left( \Mod(A)^{hG}\right)^{\mathrm{dual}} \simeq \left(\Mod(A)^{\mathrm{dual}}\right)^{hG}\simeq\Perf(A)^{hG}.\]
\end{Rem}

\begin{Not}\label{not-AX}
    Let $G$ be a finite group and $A \in \Fun(BG, \CAlg)$. By applying \Cref{con-cotensor-G-equiv} to $\cat C=\Mod(A)$ and using \Cref{prop-sep-preserves-limits}, we see that the cotensor functor extends to a functor 
    \[
    A^\bullet\colon\Fin_G\op \to \CAlg^{\sep}(\Perf(A)^{hG}), \quad X \mapsto A^X.
    \]
By the discussion in the cited construction, for any subgroup $H \subseteq G$, there is an equivalence of commutative algebras 
    \[
    A^{G/H} \simeq  (\res ^G_H A)^{\circledast_H G}.
    \]
    These commutative algebras $A^X$ migrate to any symmetric monoidal stable $\infty$-category receiving $\Perf(A)^{hG}$, like the Tate category of \Cref{def-orbit-tate}.  
\end{Not}
We next introduce the stable module category following \cite{Krause20}.

\begin{Not}\label{not-stable-module-category}
    Let $G$ be a finite group and $A \in \CAlg$. We set
    \[
    \StModfin(AG):= \frac{\Fun(BG, \Perf(A))}{\thickt{A^{G/e}}}
    \in \Cat_{\infty}^{\mathrm{st}}
    \]
    and refer to it as the \emph{stable module category of $G$ with coefficients in $A$}. We also write \[ \pi \colon \Fun(BG, \Perf(A))\to \StModfin(AG)\]
    for the quotient functor. 
\end{Not}

\begin{Exa}
    If $A=k$ is a field, then $\Ho(\StModfin(kG))$ is equivalent to the stable module category of $kG$ which appears in modular representation theory. This follows from the fact that Verdier quotients commute with taking homotopy categories and \cite[Theorem 2.1]{Rickard}.
\end{Exa}

\begin{Rem}
    We warn the reader that the  stable module category is not idempotent complete in general. Our definition agrees with that of Mathew \cite{Mathew2015} up to idempotent completion. See \cite[Remark 4.3]{Krause20} for more details. 
\end{Rem}

\begin{Rem}
As the stable module category is defined as a Verdier quotient by a thick ideal, it inherits a symmetric monoidal structure from $\Fun(BG, \Perf(A))$, see for example~\cite[Theorem I.3.6]{Nikolaus-Scholze}. It follows that $\StModfin(AG) \in \CAlg(\Cat_\infty^{\mathrm{st}})$ in such a way that the canonical functor $\pi \colon \Fun(BG, \Perf(A))\to \StModfin(AG)$ is symmetric monoidal.
\end{Rem}

\begin{Rem}\label{rem:tatemap}
    Consider $X,Y \in \Fun(BG, \Perf(A))$. Then by \cite[Lemma 4.2]{Krause20} we find that 
    \[
    \Map_{\StModfin(AG)}(X,Y)\simeq \Map_{\Perf(A)}(X,Y)^{tG}\simeq\left(\mathbb D X\otimes Y\right)^{tG}.
    \]
    As in \Cref{rem:endring}, this identifies the endomorphism ring of the unit of $\StModfin(AG)$ with $A^{tG}\in\CAlg$.
\end{Rem}

We record the following result for future reference.

\begin{Lem}\label{lem-fully-faithful}
    Let $G$ be a finite group of order $g$, and suppose that $A \in\CAlg$ is bounded below. 
    Then:
    \begin{itemize}
    \item[(a)] The canonical map $A^\wedge_g \to \prod_{0 < p | g} A_p^{\wedge}$, from the $g$-completion to the product of the $p$-completions, is an equivalence. 
    \item[(b)] The functor induced by base-change
    \[
    \beta\colon \StModfin(AG) \to \StModfin(A^{\wedge}_g G)\simeq  \prod_{p \mid g} \StModfin(A^{\wedge}_p G)
    \]
    is fully faithful.
    \end{itemize}
\end{Lem}

\begin{proof}
Part (a) is true more generally for an arbitrary spectrum $X$ in place of $A$. For any abelian group $C$ and coprime integers $a,b\ge 1$ both of the group homomorphisms
\[ C/ab \stackrel{\mathrm{red}}{\longrightarrow}C/a\oplus C/b\,\,\mbox{ and }\,\, C[ab]\stackrel{(\cdot b,\cdot a)}{\longrightarrow}  C[a] \oplus C[b]\]
are isomorphisms.
The map of fiber sequences 
\[
\begin{tikzcd}
        X \arrow[d,"b"]\ar[r,"ab"] & X \ar[d,"\mathrm{id}"] \arrow[r] & X/ab\arrow[d]\\
        X\arrow[r,"a"] & X \ar[r] & X/a
\end{tikzcd}
\]
induces a map of short exact sequences
\[
\begin{tikzcd}
        0\ar[r] & (\pi_*(X))/ab\ar[d,"\mathrm{red}"]\ar[r]& \pi_*(X/ab)\arrow[r]\arrow[d] & (\pi_*(\Sigma X))[ab] \ar[r]\ar[d,"\cdot b"]\ar[r] & 0\\
         0\ar[r] & (\pi_*(X))/a\arrow[r]& \pi_*(X/a)\arrow[r] & (\pi_*(\Sigma X))[a] \arrow[r]\arrow[r] & 0.
\end{tikzcd}
\]
Reversing the roles of $a$ and $b$ induces a similar map of short exact sequences which combine into
\[\resizebox{\columnwidth}{!}{$\displaystyle
\begin{tikzcd}[ampersand replacement=\&]
0\arrow[r] \& (\pi_*(X))/ab\arrow[d,"\mathrm{red}"]\arrow[r] \& \pi_*(X/ab)\arrow[r]\arrow[d] \& (\pi_*(\Sigma X))[ab] \arrow[r]\arrow[d,"{(\cdot b,\cdot a)}"]\arrow[r] \& 0\\ 
        0\arrow[r] \& (\pi_*(X))/a\oplus (\pi_*(X))/b \arrow[r] \& \pi_*(X/a)\oplus\pi_*(X/b)\arrow[r] \& (\pi_*(\Sigma X))[a]\oplus (\pi_*(\Sigma X))[b] \arrow[r] \& 0.
\end{tikzcd}
$}
\]
Since the outer maps are isomorphisms by our initial observation, the map $X/ab \to X/a\times X/b$ is also an equivalence. 
One then obtains the claim by induction on the number of prime factors of $g$ which obviously starts with the case that $g$ is a prime power. In the induction step, one has $g=p^ng'$ for some prime $p$ not dividing $g'$, and can apply the above with $a=p^{nk}$ and
$b=(g')^k$ for an arbitrary $k\ge 1$, obtaining
\[ X/p^{nk} \times X/(g')^k\simeq X/g^k.\]
Passage to the limit over $k$ yields $X_p^{\wedge}\times X_{g'}^{\wedge}\simeq X_g^{\wedge}$, as desired.

The proof of (b) is essentially \cite[Lemma 5.1]{Krause20}. We recall the argument here for completeness. 
    We only show that the base-change map is fully faithful, the fact that the  stable module category splits as a product is an easy consequence of the fact that $A_g^{\wedge}$ itself splits. For this, by \Cref{rem:tatemap}, we need to check that the map 
    \begin{equation}\label{map}
     (\mathbb{D}X \otimes_A Y)^{tG} \to ((A_g^{\wedge} \otimes_A \mathbb{D}X )\otimes_{A_g^{\wedge}}(A_g^\wedge \otimes_A Y))^{tG}\simeq (A_g^\wedge \otimes_A\mathbb{D}X \otimes_{A} Y)^{tG}
    \end{equation}
    is an equivalence for all $X,Y\in\Fun(BG,\Perf(A))$. Recall that by \cite[Lemma I.2.9]{Nikolaus-Scholze}, the Tate construction of any bounded below spectrum is $g$-complete \footnote{The cited result is stated for $G=C_p$ but the same proof applies to any finite group $G$.}. Thus if $A$ is bounded below, this holds for any perfect $A$-module such as $\mathbb{D}X \otimes_A Y$.
    
    In particular, the modules appearing in (\ref{map}) are $g$-complete, and so to verify that the map is an equivalence, we can do so after tensoring with $A^{tG}/g$. Using exactness of the Tate construction, this reduces to checking that the reduction mod $g$ of $A$ and $A^\wedge_g$ agree, which is clear.
\end{proof}

\begin{Rem} An alternative argument for the key point of  \Cref{lem-fully-faithful}(a) is to recall that every stable category is tensored over finite spectra, so one needs to see that for coprime integers $a$ and $b$, the map
\[ \mathbb S/ab \longrightarrow \mathbb S/a\times \mathbb S/b\]
constructed there is an equivalence. Since these spectra are connective, it suffices by the Hurewicz theorem to see that the above map in an equivalence after tensoring with $\mathbb Z$, which is true by the Chinese remainder theorem.    
\end{Rem}

Finally we introduce the last category of interest and explain how it relates to the stable module category.

\begin{Def}\label{def-orbit-tate}
Let $G$ be a finite group and $A\in \Fun(BG, \CAlg)$.
\begin{itemize}
    \item[(a)]  We define the \emph{homotopy $G$-{orbits} of $\Perf(A)$} to be
\[
\Perf(A)_{hG}:=\colim_{BG} \Perf(A)
\]
where the colimit is calculated in $\Cat_\infty^{\mathrm{perf}}$. 
\item[(b)] There is a canonical comparison map between the colimit and the limit 
\[
\mathrm{Nm}_G \colon \Perf(A)_{hG} \to\Perf(A)^{hG}
\]
which identifies with the norm map. We define the \emph{$G$-Tate of $\Perf(A)$} to be
\[
\Perf(A)^{tG}:=\cof(\Perf(A)_{hG} \xrightarrow{\mathrm{Nm}_G} \Perf(A)^{hG})\in \Cat_\infty^{\mathrm{perf}};
\]
the cofibre of the norm map in $\Cat_\infty^{\mathrm{perf}}$. Note that there is a canonical functor $q \colon \Perf(A)^{hG}\to \Perf(A)^{tG}$.
\end{itemize}
\end{Def}

We now record a more explicit description of the $G$-Tate of $\Perf(A)$. We denote by $(-)^\natural$ the idempotent completion functor.

\begin{Lem}\label{tate-description}
Let $G$ be a finite group and $A\in \Fun(BG, \CAlg)$. Then
\[
\Perf(A)^{tG}\simeq\left (\frac{\Perf(A)^{hG}}{\thickt{A^{G/e}}}\right)^\natural.
\]
\end{Lem}

\begin{proof}
Since cofibres in $\Cat_\infty^\mathrm{perf}$ are given by the idempotent completion of the Verdier quotient, we need to check that the image of $\mathrm{Nm}_G$ is $\thickt{A^{G/e}}$. By \cite[Proposition 2.18]{CMNN2020} taking $\mathcal A=\Perf(A)\in\Fun(BG,\Cat_\infty^{\mathrm{perf}})$ there, we know that the image of the norm map agrees with $\thick(A^{G/e})$ which is automatically an ideal by an application of the projection formula (see \Cref{rem-proj-formulas}).
\end{proof}

\begin{Exa}\label{ex:idempotent}
    Suppose that $A\in \Fun(BG, \CAlg)$ with trivial $G$-action. Then $\Perf(A)^{tG}$ is the idempotent completion of $\StModfin(AG)$.
\end{Exa}

\begin{Rem}
    We note that $\Perf(A)^{tG}$ acquires a symmetric monoidal structure from $\Perf(A)^{hG}$ making the functor $q\colon \Perf(A)^{hG}\to \Perf(A)^{tG}$ into a symmetric monoidal functor, see discussion after \cite[Definition 2.16]{Mathew2016}. It follows that $\Perf(A)^{tG}\in \CAlg(\Cat_\infty^{\mathrm{perf}})$ and so $\Ind(\Perf(A)^{tG})\in\CAlg(\PrL)$.
\end{Rem}

\begin{Lem}\label{res-coinduction}
    Let $H$ be a subgroup of a finite group $G$ and $A\in\Fun(BG, \CAlg)$. Then the restriction-coinduction adjunction  descends to the Tate categories
    \[
     \res^G_H\colon \Perf(A)^{tG} \rightleftarrows \Perf(A)^{tH}: (-)^{\circledast_H G}
    \]
    Furthermore, these last functors induce adjoint pairs at the level of commutative algebra objects. 
\end{Lem}

\begin{proof}
    We claim that $\res_H^G A^{G/e}\in \thickt{A^{H/e}}$. From this and the universal property of the Verdier quotient we see that the symmetric monoidal functor $\res^G_H\colon \Perf(A)^{hG}\to \Perf(A)^{hH}$ descends to a symmetric monoidal functor between the Tate categories. Since $A^{G/e}\simeq (A^{H/e})^{\circledast_H G}$ also the right adjoint passes to the Tate categories. It also follows that $(-)^{\circledast_H G}$ acquires a lax symmetric monoidal structure so the adjunction descends to commutative algebra objects. Therefore it is only left to prove the claim. By \Cref{twistedinductioncoinduction} we can instead verify that $\res_H^G( G \circledast_e A)\in \thickt{H \circledast_e A}$. By the Mackey formula (cf. \Cref{rem-Mackey-dec}), $\res_H^G (G\circledast_e A) \simeq \bigoplus_{G/H} (H\circledast_e A)$ and is therefore indeed in $\thickt{H\circledast_e A}$.
\end{proof}

\section{The indecomposable condition}

Let $\CC$ be a stable presentably\footnote{In fact, the reader can verify that additively symmetric monoidal and idempotent-complete is sufficient.} symmetric monoidal $\infty$-category with unit object $\unit$. Recall that an idempotent of a commutative algebra $A\in\CAlg(\CC)$ is by definition an idempotent element of the discrete ring $\pi_0(A):=\pi_0(\Map_{\CC}(\unit, A))$. We then say that $A$ is \emph{indecomposable} if $A$ is nonzero and if it does not decompose as a product of two nonzero rings. This is equivalent to $\pi_0(A)$ being an indecomposable ring. 
See \cite[Section 3]{NaumannPol} for more details.

\begin{Rem}\label{rem:idempotents}
Let $\CC$ be a stable presentably symmetric monoidal $\infty$-category and consider $R,S\in\CAlg(\CC)$ with $S$ indecomposable. 
For any decomposition $R=\prod_{i=1}^n R_i$ into a finite product, 
we obtain
\[ 
\Map_{\CAlg(\CC)}(R,S)=\coprod_{i=1}^n\Map_{\CAlg(\CC)}(R_i,S).
\]
\end{Rem}

We next introduce one of our conditions which we require to prove our main result.

\begin{Def}\label{def-very-good}
    Let $G$ be a finite group and consider $A\in \Fun(BG, \CAlg)$. We say that $A$ satisfies the \emph{indecomposable condition} (in short IC) if the commutative algebras $A^{hH}$ and $A^{tH}$ are indecomposable for all subgroups $1\not= H \subseteq G$.
\end{Def}

We next give a few examples of commutative algebras satisfying this condition. 

\begin{Prop}\label{prop-IC-ring}
    Let $G$ be a finite group and let $A\in \Fun(BG, \CAlg^{\heartsuit})$ be a discrete commutative algebra with trivial $G$-action. 
    Then $A$ satisfies IC if and only if $A$ and $A/|H|$ are indecomposable for all $1 \not =H\subseteq G$.
\end{Prop}

\begin{proof}
    Using the homotopy fixed points and Tate spectral sequences we see that 
    \[
    \pi_0(A^{hH})= H^0(H,A)=A \quad \mathrm{and} \quad \pi_0(A^{tH})=\widehat{H}^0(H,A)= A/|H|
    \]
    for all $H\subseteq G$.
\end{proof}

\begin{Exa}\label{exa-integers-IC}
    If $G$ is a $p$-group, then the integers $\bbZ$ and any field $k$ of characteristic $p$ satisfy IC with respect to the trivial $G$-action. 
\end{Exa}

\begin{Thm}\label{thm-IC}
 Let $G$ be a finite group and let $\mathbb{S}_G \in \Fun(BG,\CAlg)$ be the sphere spectrum with trivial $G$-action. Then $\mathbb{S}_G$ satisfies IC if and only if $G$ is a $p$-group for some prime number $p$.  \end{Thm}

\begin{proof}
    We start by calculating $\pi_0(\mathbb{S}_G^{hG})$ and $\pi_0(\mathbb{S}_G^{tG})$. By the Segal conjecture, we have
    \[
    \pi_0(\mathbb{S}_G^{hG})=\pi_0(\mathbb{S}^{BG})=\widehat{A(G)}_I
    \]
    the completion of the Burnside ring of $G$ at the augmentation ideal $I$. For the Tate construction we look at the defining fiber sequence
\[ (\mathbb{S}_G)_{hG}\xra{N} \mathbb{S}_G^{hG}=\mathbb{S}^{BG}\to \mathbb{S}_G^{tG}.\]
Since $\pi_{-1}((\mathbb{S}_G)_{hG})=0$ and $\pi_0((\mathbb{S}_G)_{hG})=\mathbb Z$, we deduce that the ring homomorphism $ \pi_0(\mathbb{S}_G^{hG})\to \pi_0(\mathbb{S}_G^{tG})$
is surjective with kernel generated by $N(1)=[G] \in \widehat{A(G)}_I$. Hence
\[
\pi_0(\mathbb{S}_G^{tG})= \widehat{A(G)}_I/([G]).
\]
Next we check when these two rings are indecomposable. To this end, recall that for any henselian pair $(R,I)$, the reduction $R\to R/I$ induced a bijection on idempotents, because the $R$-algebra $R[X]/(X^2-X)$ classifying idempotents is \'etale, see \cite[Tag 09XI]{stacks-project}. This in particular applies to the $I$-adically complete ring $\widehat{A(G)}_I$. It also applies to the quotient ring $\widehat{A(G)}_I/([G]))$ as it is a finitely generated module over the complete Noetherian ring $\widehat{A(G)}_I$, hence also $I$-adically complete. We are thus reduced to checking when the two rings
\[ \widehat{A(G)}_I/I \quad \mathrm{and} \quad \widehat{A(G)}_I/(([G])), I) \]
are indecomposable. Since the first quotient is $\mathbb Z$, and $[G]-|G|\in I$, we are reduced to checking when the two rings $\Z$ and $\Z/|G|$ are indecomposable. By the Chinese Remainder Theorem this holds if and only if $G$ is a nontrivial $p$-group for some prime $p$. It is only left to note that IC always holds if $G$ is the trivial group. 
\end{proof}

\begin{Exa}
 For a finite group $G$, consider $KU \in \Fun(BG, \CAlg)$, the spectrum of topological $K$-theory with trivial $G$-action. Then $\pi_0(KU^{hG})$ is the completion of the complex representation ring $RU(G)$ at the augmentation ideal $I$. The observation about henselian pairs in the previous proof shows that the quotient map 
  \[
  RU(G)^{\wedge}_I \to RU(G)^{\wedge}_I/ I \simeq \Z 
  \]
  induces a bijection on idempotents. Thus $KU^{hG}$ is indecomposable for every finite group $G$. On the other hand \cite[Example I.2.3.(iii)]{Nikolaus-Scholze} shows that $KU^{tG}$ is not indecomposable if $G=C_{p^2}.$
\end{Exa}

\section{The retraction condition}

In this section we isolate the following condition:

\begin{Def}\label{def-retraction condition}
    Let $G$ be a finite group and  $A\in\Fun(BG,\CAlg)$. We say that $A$ satisfies the \emph{retraction condition for} $G$ 
    if for all $H\subseteq G$ such that $A$ is a retract of $A^{G/H}$ in $\Perf(A)^{tG}$, we have $H=G$. We say that $A$ satisfies the \emph{retraction condition} (in short RC)
    if for all subgroup $H\subseteq G$, the restricted commutative algebra $\res^G_HA$ satisfies the retraction condition for $H$.
    \end{Def}

In applications, we will use the following consequence of RC.
\begin{Rem}\label{rem-RC}
    Suppose that $A\in\Fun(BG,\CAlg)$ satisfies the retraction condition for $G$. Then $A$ is a retract of $A^{G/H}$ in $\Perf(A)^{hG}$ if and only if $H=G$. The ``if'' part is clear and for the ``only if'' part note that if $A$ is a retract of $A^{G/H}$ in $\Perf(A)^{hG}$, then we can apply the canonical functor $\Perf(A)^{hG}\to \Perf(A)^{tG}$ to this retraction diagram to find, by RC, that $H=G$.
\end{Rem}

This condition lets us prove the following result.

\begin{Thm}\label{thm-ic+rc}
    Let $G$ be a finite group and let $X,Y$ be finite $G$-sets. Suppose that the isotropy of $Y$ does not contain the trivial group. Let $A\in\Fun(BG,\CAlg)$ satisfy IC and RC. 
    Then both canonical maps
    \[
    \Map_{\Fin_G}(Y,X)\xrightarrow{\sim}\Map_{\CAlg(\Perf(A)^{hG} )}(A^X, A^Y) \xrightarrow{\sim} \Map_{\CAlg(\Perf(A)^{tG})}(A^X, A^Y)
    \]
    are equivalences.
\end{Thm}

\begin{proof}
We note that if $G=e$, then $Y=\emptyset$ and so $A^Y=0$. The canonical maps in the theorem are then equivalences as all mapping spaces involved are contractible. So without loss of generality we can assume that $G$ is nontrivial.
Firstly, let us treat the special case in which $Y=G/G$ so that $A^Y=A$. Decomposing $X$ into its orbits $X \simeq \coprod_i G/H_i$, we see that $A^X \simeq \prod_i A^{G/H_i}$. Note that the unit $A$ is indecomposable both in $\Perf(A)^{hG}$ and $\Perf(A)^{tG}$ by the assumption IC for $A$. Remark \ref{rem:idempotents} then implies that the relevant mapping space of commutative algebras in either category is given by 
\[
\Map(A^{X}, A) \simeq\coprod_i \Map(A^{G/H_i}, A),
\]
and the analogous result holds obiously for $\Fin_G$.
We are therefore reduced to showing that passage to the quotient category induces an equivalence on the mapping spaces $\Map(A^{G/H}, A)$, and these are equivalent to $\Map_{\Fin_G}(*,G/H)$. Indeed, we claim these mapping spaces are simultaneously empty or contractible, depending on $H$: If $H \not =G$, then we claim that both mapping spaces are empty as there are not even any unital maps from $A^{G/H}$ to $A$. This follows from the fact that $A$ satisfies RC and \Cref{rem-RC}.
If $H=G$, then both mapping spaces are contractible as the unit $A=A^{G/G}$ is the initial commutative algebra in both categories. We have therefore proved the theorem in the case that $Y=G/G$. 

For the general case, decompose $Y$ into orbits $Y\simeq \coprod_i G/H_i$, and note that by our assumption $H_i \not =e$ for all $i$. Then in all three categories we find that 
\[
\Map(A^X, A^Y)=\prod_i \Map(A^X, A^{G/H_i})=\prod_i \Map(\res^G_{H_i}(A^X), A),
\]
where we used that $A^{G/H_i}$ is coinduced, see \Cref{not-AX}.  Then the claim follows from the previous paragraph, observing that for each $i$, we have 
\[ 
\res^G_{H_i}(A^X)\simeq \prod_j A^{H/H_{ij}},\]
where
\[ 
\res^G_{H_i}(X)\simeq \coprod_j H_i/H_{ij}
\]
is the decomposition of the $H_i$-space into orbits. (Note this argument uses that $A$ is indecomposable in the Tate category of each $H_i$).
\end{proof}

We now list some examples of commutative algebra objects which satisfies RC.

\begin{Lem}\label{RC-for-com-rings}
    Let $G $ be a finite group and let $A\in \CAlg^{\heartsuit}$ be a $|G|$-torsion free discrete commutative ring in which no prime divisor of $|G|$ is invertible. Then $A\in \Fun(BG,\CAlg)$ with trivial $G$-action satisfies the retraction condition for $G$. In fact $A$ satisfies RC.
\end{Lem}
\begin{proof}
    Suppose that $A$ is a retract of $A^{G/H}$ in $\Perf(A)^{tG}$  for some $H \subseteq G$. We need to show that $H=G$. Applying \[ \pi_0\Map_{\Perf(A)^{tG}}(A,-) \colon \Perf(A)^{tG}\longrightarrow \Mod_{\pi_0(A^{tG})}(\Ab)=\Mod_{A/|G|}(\Ab)\] 
    to a retraction diagram and using \Cref{rem:endring} and the fact that $A^{G/H}$ is a coinduced module, we find that the $A$-module $A/|G|$ is a retract of $A/|H|$. Note that $A/|H|$ is $|H|$-torsion, therefore so is $A/|G|$. It follows that $|H|\cdot 1_A = 0$ mod $|G|$, so there is $a\in A$ such that $a [G:H]|H| = |H|\cdot 1_A$. The $|G|$- and hence $|H|$-torsion-freeness then implies that $a[G:H] = 1$, so $[G:H]$ is invertible in $A$. By assumption, this forces $|H|=|G|$ and thus $H=G$. 
\end{proof}

\begin{Lem}\label{emptymaps-connective}
   Let $G$ be a finite group and let $A\to B$ be a map in $\Fun(BG, \CAlg)$. Then RC for $B$ implies RC for $A$. 
\end{Lem}

\begin{proof}
The naturality in $\cat C$ in \Cref{con-cotensor-G-equiv} implies $B\otimes_A A^{H/K}\simeq B^{H/K}$ in $\Perf(B)^{tH}$ for any subgroups $K \subseteq H \subseteq G$.
\end{proof}
\begin{Exa}\label{ex:RC-Z-S}
    Let $p$ be a prime number and $G$ be a finite $p$-group. It follows from \cref{RC-for-com-rings} that $\bbZ$ satisfies $RC$. Using this and \cref{emptymaps-connective} we deduce that $\mathbb{S}$ with trivial $G$-action satisfies $RC$ too. 
\end{Exa}

\section{Separably closed commutative algebras}

In this section we introduce the last condition required to prove our main result.

\begin{Def}\label{def-sep}
    We say that $A \in \CAlg$ is \emph{separably closed} if the cotensor functor is an equivalence 
        \[
         A^{(-)}\colon \Fin\op \stackrel{\simeq}{\longrightarrow} \CAlg^{\sep} (\Perf(A)).
        \]
        If $G$ is a finite group and $A \in\Fun(BG,\CAlg)$, then we say that $A$ is separably closed if the underlying commutative algebra $A\in \CAlg$ is separably closed.
\end{Def}

Let us give a few remarks on this definition. 
\begin{Rem}
    Any separably closed commutative algebra  $A$ is necessarily indecomposable. Indeed using that retracts of separable algebras are separable, we see that a nontrivial decomposition of $A$ would give a nontrivial decomposition of $\ast \in \Fin$, and so a contradiction.
    \end{Rem}

\begin{Rem}
    We warn the reader that our definition of separably closed commutative algebra is more restrictive than the one given by Rognes in \cite[Definition 10.3.1]{Rognes2008}. This is because any Galois extension is in particular separable, see \cite[Lemma 9.1.2]{Rognes2008}. To see that our definition is strictly stronger, we fix a separably closed field $k$ of characteristic $p$, and consider $A=k^{tE}$ for an elementary abelian $p$-group $E$. Then $A$ is separably closed in the sense of Rognes by work of Mathew~\cite[Theorem 9.9]{Mathew2016}. On the other hand by \cite[Theorem 2.30]{Mathew2016}, the $\infty$-category $\Perf(k^{tE})$ identifies with the small stable module category of $kE$, and in there we know that there are separable algebras which are not just products of $k$, for instance $k[E/E_0]$ for $e\not =E_0\subset E$, as one checks by contemplating support in the Balmer spectrum.
\end{Rem}

Let us give a few examples.

\begin{Lem}
    Let $S \in \CAlg$ be separably closed and let $R:=\infl_e^G S \in\CAlg(\Sp_G)$. Then, for all subgroups $K \subseteq G$, the geometric fixed points $\Phi^K R$ are separably closed. 
\end{Lem}

\begin{proof}
    For every $K \subseteq G$, there is an equivalence $\Phi^K \infl_e^G S \simeq S$.
\end{proof}

\begin{Exa}\label{ex-sphere-sepclosed}
    For instance we can take $S= \mathbb{S}$ so that $R= \mathbb{S}_G$ or $S=\Z$ so that $R=\infl_e^G(\Z)$. The fact that the rings $S$ are separably closed follows from Minkowski's theorem, see~\cite[Proposition 10.5]{NaumannPol}.
\end{Exa}

Next, we record a consequence of this condition. 

\begin{Lem}\label{lem-separablyclosed-G}
    Let $G$ be a finite group and $A\in\Fun(BG, \CAlg)$ separably closed. Then the functor from \Cref{not-AX}
    \[
    A^{(-)}\colon \Fin\op_G \xrightarrow{\sim} \CAlg^{\sep} (\Perf(A)^{hG})
    \] 
    is an equivalence. 
\end{Lem}

\begin{proof}
     Our assumption that $A$ is separably closed ensures that 
     \[
      A^{(-)}\colon \Fin\op \xrightarrow{\sim} \CAlg^{\sep}(\Perf(A))
     \]
     is an equivalence. The claim then follows from this by taking $G$-homotopy fixed points. 
\end{proof}

\begin{Not}\label{nota-groupoids}
Let $\cat F$ be a family of subgroups of a finite group $G$, i.e. a possibly empty set of subgroups of $G$, closed under sub-conjugation. We denote by $\Fin_{G}/\cat F$ the full subcategory of $\Fin_G$ spanned by those finite $G$-sets $X$ such that for all $x\in X$, the stabilizer subgroup Stab$_G(x)\subseteq G$ is not in $\mathcal F$. We also denote by $\Fin_{G}^{\mathrm{free}}$ the category of finite free $G$-sets.  To simplify notation we we will write 
    \begin{align*}
    \CFin_G & := \Fin_G^{\simeq} \\
    \CFin_G/\cat F & :=(\Fin_G/\cat F)^\simeq \\ \CFin_{G}^{\mathrm{free}} & :=(\Fin_{G}^{\mathrm{free}})^{\simeq}\\
    \end{align*}
    using caligraphic letters to indicate the maximal groupoid underlying a category. We note that the above groupoids canonically carry a cocartesian monoidal structure turning them into symmetric monoidal groupoids, or equivalently, into $\mathbb{E}_\infty$-monoids in spaces.
\end{Not}
\begin{Rem}\label{F-splitting}
     There is a symmetric monoidal (wrt disjoint union) equivalence of groupoids
    \begin{equation}\label{eq-decomposition}
    \CFin_G/\cat F \simeq \prod_{(H)\not \in \cat F} \CFin_{W_G(H)}^{\mathrm{free}}, \quad X \mapsto (\mathrm{Map}_G^{\mathrm{inj}}(G/H, X))_{(H)}
    \end{equation}
    where the product is taken over a set of representatives of $G$-conjugacy classes of subgroups $H \not \in \cat F$, and $\mathrm{Map}_G^{\mathrm{inj}}$ denotes the set of $G$-equivariant injective maps. 
\end{Rem}

\begin{Cons}\label{cons:modF}
      For any two families $\cat F\subseteq\cat F'$, the fully faithful inclusion $\iota \colon \Fin_G/\cat F'\subseteq\Fin_G/\cat F$
    admits a right adjoint (hence a colocalization) 
    \[ \pi\colon \Fin_G/\cat F\longrightarrow\Fin_G/\cat F'\]
    with counit 
    \[ \pi(X)=\{ x\in X\,|\, \mathrm{Stab}_G(x)\not\in\cat F'\} \subseteq X.\]

    We also write $X/\mathcal F$ for $\pi(X)$. 
\end{Cons}

\begin{Prop}\label{prop:decomp_Fin_G}
Let $G$ be a finite group, $\cat F$ a family of subgroups of $G$ and $K\subseteq G$ a subgroup such that 
$K\not\in\cat F$, but for all proper subgroups $K'\subsetneq K$ we have $K'\in\cat F$.
Then in the solid diagram
\[ \begin{tikzcd}
        \Fin_G/\cat F\arrow[r,"\pi"]\arrow[d, "(-)^K"'] & \Fin_G/(\cat F\cup K)\arrow[d, dotted, "\Phi"] \\
        \Fin_{W_G(K)} \arrow[r, "\pi'"] & \Fin_{W_G(K)}/\{ e\}
    \end{tikzcd}
\]
 there is a unique factorization $\Phi$ of $\pi'\circ (-)^K$ through $\pi$. In fact $\Phi$ is given by the composite 
 \[
 \Fin_G/(\cat F\cup K) \xrightarrow{\iota} \Fin_G/\cat F \xrightarrow{(-)^K} \Fin_{W_G(K)}\xrightarrow{\pi'} \Fin_{W_G(K)}/\{ e\}.
 \]
\end{Prop}
We note that while both horizontal functors have left adjoints, this square is generally not left adjointable. In fact, it is exactly when it is a pullback square, which happens when every subgroup has a strictly larger normalizer in $G$. It however follows from \Cref{F-splitting} that it is always a pullback when restricting to maximal subgroupoids.
\begin{proof}
Since $\pi$ is a colocalization, the claim about the factorization is equivalent to the functor $\pi'\circ(-)^K$
inverting the counit
\[ \iota\colon\pi(X)=\{ x\in X\,|\, \mathrm{Stab}_G(x)\not\in\cat F\cup K\} \subseteq X, \, X\in \Fin_G/\cat F.\]
Now, $\pi'(\iota^K)$ computes to be the inclusion
\[ \{ x\in X^K\,|\,  \mathrm{Stab}_G(x)\not\in\cat F\cup K\mbox{ and }\mathrm{Stab}_{W_G(K)}(x)\neq\{ e\}\} \subseteq \{ x\in X^K\, |\, \mathrm{Stab}_{W_G(K)}(x)\neq\{ e\}\},\]
which is immediately checked to be an equality. The resulting formula for $\Phi$ is a formal consequence of the above discussion. 
\end{proof}

Recall the functor $q$ of passage to the Tate-category from \Cref{def-orbit-tate}.
\begin{Lem}\label{factorizationgeneral}
    Let $G$ be a finite group and assume $A\in \Fun(BG, \CAlg)$ satisfies IC and RC. Then, there is a unique factorization $\iota$ 
    \[
    \begin{tikzcd}
        \CAlg^{\sep} (\Perf(A)^{hG})\arrow[r,"q"] & \CAlg^{\sep} (\Perf(A)^{tG}) \\
        \Fin\op_G\arrow[u,"A^{(-)}"] \arrow[r, "\pi"] & \left( \Fin_G/\{ e\} \right)\op \arrow[u, hook, dotted,"\mathrm{\iota}"']
    \end{tikzcd}
    \]
   and $\iota$ is fully faithful.
\end{Lem}

\begin{proof}To get the factorization, we need to see that $q\circ A^{(-)}$ inverts the counit
\[ \pi(X)=\{ x\in X\,|\, \mathrm{Stab}_G(x)\neq\{ e\}\}\subseteq X\,,\,X\in\Fin_G.\]
This is clear since $q(A^G)=0$.
Full faithfulness of $\iota$  follows from \Cref{thm-ic+rc}.
\end{proof}

\section{Pullback decomposition of equivariant spectra}
In this section we recall the pullback decomposition for equivariant spectra given in \cite[Section 7]{Glueingpaper} which will be the key tool for classifying separable algebras in this setting. Here we will be brief and refer the interested reader to the above reference for more details.
A similar pullback decomposition already appears in \cite[Theorem 3.11]{Krause20} for inflated commutative ring spectra; we refer the reader to the 
 introduction of \cite{Glueingpaper} for a more in depth discussion of the relation between these two results. 

Fix a finite group $G$ and a family $\cat F$ of subgroups of $G$. The case $\cat F=\emptyset$ is allowed here. Given a family $\cat F$ and a subset $S$ of the set of subgroups of $G$, we let $\cat F\cup S$ denote the smallest family containing $\cat F$ and $S$. We will exclusively be concerned with the case $S=\{K\}$ where $K \subseteq G$ is a subgroup such that $K\notin\cat F$, but for all proper subgroups $K'\subsetneq K$ we have $K'\in\cat F$.

We denote by $\Sp_{G}/\cat F$ the $\infty$-category of genuine $G$-spectra with isotropy outside $\cat F$. Specifically, this is the smashing localization of $\Sp_G$ corresponding to the idempotent commutative algebra $\widetilde{E}\cat F$ where we have omitted the suspension spectrum functor from the notation. The canonical map $\widetilde{E}\cat{F}\to\widetilde{E} (\cat F\cup K)$ induces a further localization
\[ 
\Sp_{G}/\cat F\to \Sp_G/\cat F\cup {K}
\]
which annihilates $ \widetilde{E}\cat{F} \otimes G/K_+$. 

\begin{Not}
Fix $R\in\CAlg(\Sp_G)$ and write 
\[
\Mod_{\Sp_G}(R) \quad \mathrm{and} \quad \Mod_{\Sp_G}(R)/ \cat F
\]
for the $\infty$-category of $R$-modules in $\Sp_G$ and $\Sp_{G}/\cat F$ respectively. More concretely, we have $\Mod_{\Sp_G}(R)/\cat F\simeq \Mod_{\Sp_G}(R \otimes  \widetilde{E}\cat F)$.
We denote by 
\[
L_{\cat F \cup K}\colon \Mod_{\Sp_G}(R)/\cat F\to \Mod_{\Sp_G}(R)/\cat F \cup K
\]
the induced smashing localization which annihilates $R \otimes \widetilde{E}\cat{F} \otimes G/K_+$. 
\end{Not}
Let $\Phi^K \colon \Sp_G \to \Sp$ denote the $K$-geometric fixed points functor, and let 
\[
\Phi^K_R\colon \Mod_{\Sp_G}(R)\to \Mod(\Phi^K(R))
\]
denote its $R$-linear extension. Recall that $\Phi^K(R)$ is acted on by the Weyl-group $W_G(K)$ of $K$ in $G$, allowing us to form 
\[
\Perf(\Phi^K(R))^{hW_G(K)}.
\]
As discussed in \cite[Proposition 7.7]{Glueingpaper}, the $R$-linear geometric $K$-fixed points functor factors through some refined variant $\widetilde\Phi_R^K$ which remembers the Weyl-group action, as follows:
\[
\Phi_R^K \colon \Mod_{\Sp_G}(R) \xrightarrow{\widetilde{\Phi}_R^K} \Mod(\Phi^K( R))^{hW_G(K)}\xrightarrow{\mathrm{fgt}} \Mod(\Phi^K(R)).
\]
Moreover, using that $\widetilde{\Phi}_R^{K}$ annihilates the kernel of $L_{\cat F \cup K}$, we get an induced functor 
\[
(\overline{\Phi}^K_R)^\omega \colon (\Mod_{\Sp_G}(R)/\cat F \cup K)^\omega \to \Perf(\Phi^K R)^{tW_G(K)}.
\]
Finally, recall the functor $q$ from \Cref{def-orbit-tate}.

\begin{Thm}[{\cite[Corollary 7.12]{Glueingpaper}}]\label{cor-pullback-perfect-G-spectra}
    Let $G$ be a finite group and $R\in \CAlg(\Sp_G)$. Let $\cat F$ be a family of subgroups of $G$ and fix a subgroup $K \subseteq G$ such that $K\notin\cat F$, but for all proper subgroups $K'\subsetneq K$ we have $K'\in\cat F$. Then the following is a pullback of symmetric monoidal stable $\infty$-categories:
\[
\begin{tikzcd}
    (\Mod_{\Sp_G}(R)/\mathcal{F})^\omega \arrow[d,"(\widetilde{\Phi}_R^K)^\omega"'] \ar[r,"L^\omega_{\mathcal{F}\cup K}"] 
   &   (\Mod_{\Sp_G}(R)/\mathcal{F}\cup K )^\omega\arrow[d, "(\overline{\Phi}^K_R)^\omega "]\\
    \Perf(\Phi^K R)^{hW_G(K)}\arrow[r, "q"]& \Perf(\Phi^K R)^{t W_G(K)}.
\end{tikzcd}
\]
\end{Thm}

\begin{Cor}\label{thm-pullback-stmodfin}
     Let $G$ be a finite group, $\cat F$ a family of subgroups of $G$ and let $K\subseteq G$ be a subgroup such that $K \not \in \cat F$ and $K'\in \cat F$ for all $K'\subsetneq K$. Suppose we are given $R \in \CAlg(\Sp_G)$ such that $\Phi^K(R)$ is bounded below and $W_G(K)$ acts trivially on $\Phi^K(R)$. Then the following is a pullback of symmetric monoidal stable $\infty$-categories
     \[
    \begin{tikzcd}
        (\Mod_{\Sp_G}(R)/\cat F)^\omega \arrow[r,"L_{\cat F \cup K}^\omega"] \arrow[d,"(\widetilde{\Phi}_R^K)^\omega"'] & (\Mod_{\Sp_G}(R)/\cat F\cup K)^\omega\arrow[d,"\beta_0 \circ (\overline{\Phi}^K_R)^\omega"] \\
       \Fun(BW_G(K),\Perf(\Phi^K(R))) \arrow[r,"\beta_0 \circ  q"] & \prod_{0 < p \mid |W_G(K)|}\Perf(\Phi^K(R)^{\wedge}_p)^{t W_G(K)}
    \end{tikzcd}
    \]
    where $\beta_0$ is induced by the base-change functors.
\end{Cor}
\begin{proof}
    We claim that the base-change functors assemble into a fully faithful functor
    \[
    \beta_0 \colon \Perf(\Phi^K R)^{tW_G(K)} \to \prod_p \Perf((\Phi^K R)^\wedge_p)^{tW_G(K)}.
    \]
    Indeed the above functor can be obtained from \cref{lem-fully-faithful}(b) by applying the idempotent completion functor, and by identifying the result with the Tate categories via \cref{ex:idempotent}. We observe that the square in the corollary differs from \cref{cor-pullback-perfect-G-spectra} by the bottom right corner. Since we have a fully faithful functor relating the two different vertices, we obtain the pullback of the two squares are equivalent, proving the corollary.
\end{proof}

\section{The classification result}

Given a finite $G$-set $X$, we can form the commutative algebra $\mathbb{D}(\Sigma_+^\infty X) \in \CAlg(\Sp_G)$. Since separability can be checked in the homotopy category, it follows from \cite[Theorem 1.1]{Balmer_Ambrogio_Sanders} that the commutative algebra $\mathbb{D}(\Sigma_+^\infty X)$ is separable in $\Sp_G$. Given $R \in\CAlg(\Sp_G)$, there is a base change functor 
\[
R \otimes - \colon\Sp_G \to \Mod_{\Sp_G}(R)
\]
which is symmetric monoidal. We let $R^X$ denote the image of $\mathbb{D}(\Sigma_+^\infty X)$ under this functor, and observe that $R^X$ is a separable commutative algebra which is compact as an $R$-module. We therefore obtain a functor
\begin{equation}\label{betterfunctor}
R^{(-)}\colon \Fin\op_G \to \CAlg^{\sep}(\Mod_{\Sp_G}(R)^\omega).   
\end{equation}
More generally, given a family $\cat F$ of subgroups of $G$, we have a functor $R^{(-)}_{/\cat F} $:

\begin{equation}\label{eq:betterfunctor}
(\Fin_G/\cat F)\op \subseteq\Fin_G\op\stackrel{R^{(-)}}{\longrightarrow}   \CAlg^{\sep} (\Mod_{\Sp_G}(R)^\omega)\longrightarrow\CAlg^{\sep} ((\Mod_{\Sp_G}(R)/\cat F)^\omega), 
\end{equation}
\[ X \mapsto R \otimes \widetilde{E}\cat F \otimes \mathbb{D}(\Sigma_+^\infty X). \]
Note that $\eqref{eq:betterfunctor}$ generalizes $\eqref{betterfunctor}$ because $R^{(-)}_{/\emptyset}=R^{(-)}$.
We now give conditions under which this functor is fully faithful.

\begin{Prop}\label{prop-fully-faulful}
 Let $G$ be a finite group and $\cat F$ a family of subgroups of $G$. Consider $R \in\CAlg(\Sp_G)$ and suppose that for all subgroups $H \subseteq G$ with $H \not \in \cat F$, the commutative ring $\pi_0^H(R)$ is indecomposable, and $\pi_0(\Phi^H R)$ is nonzero. Then the functor \eqref{eq:betterfunctor} is fully faithful.
\end{Prop}
\begin{proof}
 We argue essentially as in the proof of \cref{thm-ic+rc}, namely we work by induction on $G$. 
We consider the map 
\[
\Map_{\Fin_G/\cat F}(Y,X)\to \Map_{\CAlg(\Mod_{\Sp_G}(R)/\cat F)}(R^X,R^Y)
\]
which we want to show is an equivalence. Since our functor sends coproducts to products, we may assume without loss of generality that $Y$ is an orbit, say $G/H$, with $H\notin \cat F$.  We then have compatible equivalences  $$\Map_{\Fin_G/\cat F}(Y,X)\simeq \Map_{\Fin_H/\cat F_{ H}}(*, \res^G_H(X))$$ 
and 
$$\Map_{\CAlg(\Mod_{\Sp_G}(R))/\cat F)}(R^X,R^Y)\simeq \Map_{\CAlg(\Mod_{\Sp_H}(R)/\cat F_H)}(R^{\res^G_H(X)},R)$$

where $\cat F_H$ is the family of subgroups of $H$ in $\cat F$. We may assume that $\res^G_H(X)$ has no isotropy in $\cat F_H$ since if $X_0\subset X$ is the subset of points with isotropy \emph{not} in $\cat F_H$, the map $\Map_{\Fin_G}(Y,X_0)\to \Map_{\Fin_G}(Y,X)$
is a bijection (any point in the image of a map $Y\to X$ has isotropy conjugate to a group containing $H$, and $H$ is not in $\cat F$). 
Thus, by induction we are reduced to the case $Y=*$. Now, since $\pi_0^G(R)$ is indecomposable, we may also assume that $X$ is an orbit (since coproducts of $X$'s pull out on either side to coproducts). Now if $X=G/G$, both sides are contractible since $R$ is the unit in $\Mod_{\Sp_G}(R)$, and if $X=G/H$ for $H$ a proper subgroup, both sides are empty, since $\Phi^G(R^X) = 0$ and $\Phi^G(R)$ is nonzero. 
\end{proof} 
In order to discuss when the functor \eqref{eq:betterfunctor} is essentially surjective, it will be useful to consider the induced functor on groupoids
\begin{equation}\label{eq:standard}
R^{(-)}_{/\cat F} \colon (\CFin_G/\cat F)\op\longrightarrow\CSep ((\Mod_{\Sp_G}(R)/\cat F)^\omega), 
\end{equation}
\[ X \mapsto R \otimes \widetilde{E}\cat F \otimes \mathbb{D}(\Sigma_+^\infty X). \]

\begin{Ter}
We refer to the separable algebras in the essential image of (\ref{eq:standard}) as {\em standard}.
\end{Ter}

Now assume that $K\subseteq G$ is a subgroup such that $K \not \in \cat F$ and $K'\in \cat F$ for all $K'\subsetneq K$. Recall from \cref{cons:modF} that the canonical inclusion $i_K \colon \CFin_{G}/\cat F \cup K \to \CFin_G/\cat F$ admits a right adjoint at the level of categories, and hence we get a retraction at the level of groupoids: 
\[
\pi_K \colon  \CFin_G/\cat F\to \CFin_G/\cat F\cup K 
\]
Definitionally, it sends the orbit $G/K$ to the empty set. 

\begin{Lem}\label{lem-top-face-cube}
    In the situation above, we have $R_{/\cat F \cup K}^{(-)}\simeq L_{\cat F \cup K} R_{/\cat F}^{i_K(-)}$ and the following diagram commutes
    \[
    \begin{tikzcd}
        \CFin_G\op/\cat F \arrow[d,"R_{/\cat F}^{(-)}"']\arrow[r,"\pi_K\op"] & \CFin_G\op/\cat F \cup K \arrow[d,"R^{(-)}_{/ \cat F \cup K}"] \\
        \CSep((\Mod_{\Sp_G}(R)/\cat F)^\omega) \arrow[r,"L_{\cat F \cup K}"] & \CSep((\Mod_{\Sp_G}(R)/\cat F \cup K)^\omega).
    \end{tikzcd}
    \]
\end{Lem}

\begin{proof}
   The first claim follows by unravelling the definitions and using the equivalence $ \widetilde{E}(\cat F \cup K) \xrightarrow{\sim}\widetilde{E}\cat F\otimes \widetilde{E}(\cat F \cup K)$. For the commutativity of the square we note that by the same argument as in \Cref{factorizationgeneral}, there is a factorization of $L_{\cat F\cup K}\circ R_{/\cat F}^{(-)}$ through $\pi_K\op$, and by the first part of the lemma, this factorization must be through $R_{/\cat F \cup K}^{(-)}$. 
   
\end{proof}
As already discussed in the introduction, our goal in this paper is to find conditions to ensure that the functor \eqref{eq:standard} is an equivalence. The next result enables an induction on $\mathcal F$.

\begin{Thm}\label{thm-base-case}
   Let $G$ be a finite group, $\cat F$ a family of subgroups of $G$ and let $K\subseteq G$ be a subgroup such that $K \not \in \cat F$ and $K'\in \cat F$ for all $K'\subsetneq K$. Consider $R \in \CAlg(\Sp_G)$ and assume that $\Phi^K(R)\in \Fun(W_G(K), \CAlg)$ satisfies IC, RC and is separably closed. Assume furthermore that the functor 
     \[
    R^{(-)}_{/\cat F \cup K}\colon (\CFin_G/\cat F \cup K)\op \xrightarrow{\sim} \CSep((\Mod_{\Sp_G}(R)/\mathcal{F}\cup K)^\omega)
    \]
    is an equivalence. Then the functor
    \[
    R^{(-)}_{/\cat F} \colon (\CFin_G/\cat F)\op \xrightarrow{\sim} \CSep((\Mod_{\Sp_G}(R)/\mathcal{F})^\omega)
    \]
    is also an equivalence.   
\end{Thm}

\begin{proof}
For the duration of this proof we set $\Gamma:=W_G(K)$. The core of the proof is the following diagram which we will explain:
\[\resizebox{\columnwidth}{!}{$\displaystyle
\begin{tikzcd}[ampersand replacement=\&]
\&\& {(\mathcal{F}\mathrm{in}_G/\mathscr{F}\cup K)^{\mathrm{op}}} \&\&\&\&\& {\CSep((\Mod_{\Sp_G}(R)/\mathcal{F}\cup K)^\omega) } \\
	{(\mathcal{F}\mathrm{in}_G/\mathscr{F})^{\mathrm{op}}} \&\&\&\&\& {\CSep((\Mod_{\Sp_G}(R)/\mathcal{F})^\omega)} \\
	\\
	\&\& {\prod_{e \not = (U)\subseteq \Gamma} (\mathcal{F}\mathrm{in}_{W_{\Gamma}(U)}^{\mathrm{free}})^{\mathrm{op}}} \&\&\&\&\& {\CSep (\Perf(\Phi^{K}R)^{t\Gamma})} \\
	\CFin_{\Gamma}^{\mathrm{op}}\simeq {\prod_{(U)\subseteq \Gamma} (\mathcal{F}\mathrm{in}_{W_{\Gamma}(U)}^{\mathrm{free}})^{\mathrm{op}}} \&\&\&\&\& {\CSep (\Perf(\Phi^{K}R)^{h\Gamma}) } \\
	\&\& \ast \\
	{(\mathcal{F}\mathrm{in}_{\Gamma}^{\mathrm{free}})^{\mathrm{op}}}
	\arrow["{R_{/\cat F \cup K}^{(-)}}", "\sim"', from=1-3, to=1-8]
	\arrow["{\overline{(-)^K}}"', {pos=0.6},dashed, from=1-3, to=4-3]
	\arrow["{\overline{\Phi}^K}", from=1-8, to=4-8]
	\arrow["{\pi_K}", from=2-1, to=1-3]
	\arrow["{R_{/\cat F}^{(-)}}"{pos=0.6}, from=2-1, to=2-6]
	\arrow["{(-)^K}"', from=2-1, to=5-1]
	\arrow["{L_{\cat F \cup K}}"', from=2-6, to=1-8]
	\arrow["{\Phi^K}",{pos=0.4}, from=2-6, to=5-6]
	\arrow["\iota", {pos=0.4}, tail, from=4-3, to=4-8]
	\arrow[dashed, from=4-3, to=6-3]
	\arrow["{\mathrm{proj}}", from=5-1, to=4-3]
	\arrow["{(\Phi^K R)^{(-)}}"{pos=0.6}, "\sim"'{pos=0.6},from=5-1, to=5-6]
	\arrow["{\mathrm{proj}}"', from=5-1, to=7-1]
	\arrow["q"',from=5-6, to=4-8]
	\arrow[from=7-1, to=6-3]
\end{tikzcd}
$}
\]
The right most face of the cube is obtained by applying $\CSep(-)$ to the pullback square of symmetric monoidal stable $\infty$-categories of \cref{cor-pullback-perfect-G-spectra}. In particular, this is a pullback square. The left face of the cube is obtain from \cref{prop:decomp_Fin_G} by passing to groupoids. In particular, the square is commutative and the functor $\overline{(-)^K}$ is given by the composite $\proj((i_K(-))^K)$. We now claim that this left face of the cube is a pullback. This follows from the pasting lemma, \cite[Proposition 4.4.2.1]{HTT}, since the bottom left square and the outer left square are pullback squares, for the latter use \Cref{F-splitting}.

Next, we explain how the remaining faces of the cube arise and why they commute. The top and bottom faces are given by the squares in \Cref{lem-top-face-cube} and \Cref{factorizationgeneral}, so they commute. In particular we deduce that $\iota$ is fully faithful and that the functor $(\Phi^K R)^{(-)}$ is an equivalence, as depicted in the diagram. The front face of the cube commutes since 
\begin{align*}
    \Phi^K R_{/\cat F}^{(X)} & =\Phi^K(R \otimes \widetilde{E}\cat F \otimes \mathbb{D}(\Sigma_+^\infty X)) \\
    &\simeq \Phi^K (R) \otimes \mathbb{D}(\Sigma_+^\infty X^K) \\
    & \simeq\iHom(\Sigma_+^\infty X^K, \Phi^K (R)) \\
    & \simeq \prod_{X^K} \Phi^K(R)\\
    & \simeq (\Phi^K R)^{X^K}
\end{align*}
where we used the properties of the geometric fixed points, dualizability of $\Sigma_+^\infty X^K$, the properties of the cotensor functor. For the back face of the cube we use the first claim in \Cref{lem-top-face-cube} and the commutativity of the other faces to compute:
\begin{align*}
    \overline{\Phi}^K R^X_{/\cat F \cup K} & \simeq \overline{\Phi}^K L_{\cat F \cup K} R^{i_K(X)}_{/\cat F} \\
    & \simeq q\Phi^K R^{i_K(X)}_{\cat F} \\
    & \simeq q (\Phi^K R)^{i_K(X)^K}\\
    & \simeq \iota \proj(i_K(X)^K)=\iota (\overline{X^K})
\end{align*}
where in the last step we used the definition of $\overline{(-)^K}$.

Now that we have shown that the above diagram commutes we are ready to prove that $R_{/\cat F}^{(-)}$ is an equivalence. Note that $R_{/\cat F}^{(-)}$ is fully faithful as it is a pullback of fully faithful functors. For essential surjectivity note that any object in $\CSep((\Mod_{\Sp_G}(R)/\cat F)^\omega)$ can be represented by a triple $(A,A', f)$ where $A\in \CSep(\Perf(\Phi^K R)^{h\Gamma})$, $A' \in \CSep((\Mod_{\Sp_G}(R)/\cat F \cup K)^\omega)$ and $f$ is an equivalence between the images of $A$ and $A'$ in $\CSep(\Perf(\Phi^K R)^{t\Gamma})$. By our assumptions we can find preimages $X\in\CFin_{\Gamma}^{\mathrm{op}}$ and $X'\in (\CFin_G/{\cat F \cup K})^{\mathrm{op}}$ of $A$ and $A'$ respectively. Moreover using the commutativity  of the back face of the cube and the fully faithfulness of $\iota$, we see that there is a unique equivalence $g$ between the images of $X$ and $X'$, which hits $f$. Then the triple $(X,X',g)$ defines an object of $(\CFin_G/\cat F)\op$ which hits $(A,A', f)$ concluding the proof.
\end{proof}

\begin{Cor}\label{main-theorem}
    Let $G$ be a finite group, $\cat F$ a family of subgroups of $G$ and consider $R \in \CAlg(\Sp_G)$. Suppose that $\Phi^K(R)\in \Fun(W_G(K), \CAlg)$ satisfies IC, RC and is separably closed for all subgroups $K \not \in \cat F$. 
    Then the functor 
    \[
    R_{/\cat F}^{(-)}\colon (\CFin_G/\cat F)\op \xrightarrow{\sim} \CSep((\Mod_{\Sp_G}(R)/\cat F)^\omega).
    \] 
    is an equivalence of $\infty$-groupoids.
\end{Cor}

\begin{proof}
    Firstly, we observe that the claim is trivial if $\cat F=\mathrm{All}$, so let us assume that $\cat F$ is a proper family. Choose an exhaustive filtration of families of subgroups of $G$
    \[
    \cat F=\cat F_0 \subseteq \cat F_1 \subseteq \ldots \subseteq \cat F_{i-1} \subseteq \cat F_{i}\subseteq \ldots \subseteq \cat F_n=\mathrm{All}
    \]
    with each $\cat F_{i}\backslash \cat F_{i-1}= \{(K_i)\}$ consisting of a single conjugacy class. Note that this implies that for all proper subgroups $K \subsetneq K_i$ we have $K \in \cat F_{i-1}$.  
    We will prove by descending induction on $i\leq n-1$, that the functor
     \[
    R_{/\cat F_i}^{(-)}\colon (\CFin_G/\cat F_i)\op \xrightarrow{\sim} \CSep((\Mod_{\Sp_G}(R)/\cat F_i)^\omega)
    \]
   is an equivalence. The case $i=0$ will then be our claim. If $i=n-1$, then $\cat F_i=\cat F_{n-1}$ is the family of proper subgroups, so taking $G$-geometric fixed points induces an equivalence $\Phi^G \colon (\Mod_{\Sp_G}(R)/\cat F_{n-1})^\omega\xrightarrow{\sim} \Perf(\Phi^G( R))$. Since $\Phi^G(R)$ is separably closed, we also deduce that
    \[
    \CSep((\Mod_{\Sp_G}(R)/\mathcal{F}_{n-1})^\omega)=\CSep(\Perf(\Phi^G(R)))=\CFin\op.
    \]
    By definition, any $X \in (\CFin_G/\cat F_{n-1})\op$ is such that $X^G=X$, and so taking $G$-fixed points gives an equivalence $\CFin_G/\cat F_{n-1} \simeq \CFin$. It then follows that the composite 
    \[
    (\CFin_G/\cat F_{n-1})\op \xrightarrow{R_{/\cat F_{n-1}}^{(-)}} \CSep((\Mod_{\Sp_G}(R)/\cat F_{n-1})^\omega)\xrightarrow[\sim]{\Phi^G} \CSep(\Perf(\Phi^G R)) \simeq \CFin\op
    \]
    is an equivalence. By the $2$-out-of-$3$ property we conclude that $R_{/\cat F_{n-1} }^{(-)}$ is an equivalence proving the base case. The inductive step then follows from \Cref{thm-base-case} and so we are done by induction.
\end{proof}

We record the two main examples we have. In the following we abbreviate 
\[
\bbZ_G:=\infl_e^G \bbZ \in \CAlg(\Sp_G)
\]
and recall that its category of modules can be identified with the $\infty$-category of derived Mackey functors of Kaledin, see \cite[Section 4]{PSW}.

\begin{Cor}\label{cor:sep-in-p-group}
    Let $p$ be a prime number and $G$ a $p$-group. Then the functor \eqref{betterfunctor} induce equivalences  
    \[
    \Fin_G\op\xrightarrow{\sim} \CAlg^{\sep}(\Sp_G^\omega) \quad \mathrm{and}\quad
    \Fin_G\op\xrightarrow{\sim} \CAlg^{\sep}(\Mod_{\Sp_G}(\bbZ_G)^\omega).
    \]
    In particular, every separable commutative algebra in either $\Sp_G^\omega$ or $\Mod_{\Sp_G}(\bbZ_G)^\omega$  is standard.
\end{Cor}

\begin{proof}
    If $G$ is a $p$-group, then for all $K\subseteq G$, the $K$-geometric fixed points are 
     $\Phi^K \mathbb{S}_G=\mathbb{S}$ and $\Phi^K \bbZ_G=\bbZ$ both with trivial Weyl group action. Note that these commutative algebras are all separably closed by \cite[Corollary 10.6]{NaumannPol} again. They satisfy IC by \Cref{thm-IC} and \Cref{exa-integers-IC}, and RC by \cref{ex:RC-Z-S}. Thus the fact that the above functors are essentially surjective follow from \Cref{main-theorem} applied to $\cat F=\emptyset$. For fully faithfulness, we apply \cref{prop-fully-faulful} and note that for any subgroups $H\subseteq G$, the groups $\pi_0^H(\Sphere_H)$ and $\pi_0^H(\bbZ_H)$ agree with the Burnside ring $A(H)$, see \cite[Remark 2.20]{PSW} for the second category, which is indecomposable by a result of Dress, see \Cref{rem-solvable}.
\end{proof}

We also describe the Galois group of these categories following \cite{Mathew2016}.

\begin{Cor}\label{cor-galois}
    Let $p$ be a prime number and $G$ a $p$-group. Then the Galois groups of $\Sp_G$ and $\Mod_{\Sp_G}(\bbZ_G)$ are trivial.
\end{Cor}

\begin{proof}
    To treat both cases simultaneously, let us write $A^G_H$ both for the separable commutative algebra $\mathbb{D}(G/H_+)$ in $\Sp_G^\omega$ and for $\bbZ_G \otimes\mathbb{D}(G/H_+) \in \Mod_{\Sp_G}(\bbZ_G)^\omega $. We have already noted that $\pi_0(\unit)$ in both categories is given by the Burnside ring $A(G)$ which is indecomposable.  It follows that their Balmer spectra are connected by \cite[Proposition 3.7]{NaumannPol}. Since the finite Galois covers are exactly those separable algebras with underlying compact object and with locally constant (i.e. constant) degree function (\cite[Corollary 8.13]{NaumannPol}), using \Cref{cor:sep-in-p-group} it will suffice to check that the degree function of $A^G_H$ is constant if and only if $G=H$. But in both cases the degree function is nonzero exactly at points in $\supp(A^G_H)$ which is a proper subset of the Balmer spectrum unless $G=H$, see \cite[Lemma 2.12 and Proposition 2.13]{PSW} and \cite[Corollary 4.13]{BS2017}.
\end{proof}

\section{Not all separable algebras are standard}\label{sec-nonall}

In this section we show that not all separable commutative algebras in $\Sp_G^\omega$ are standard. We start off with the following:
\begin{Rem}\label{rem-solvable}
Consider $\mathbb{S}_G \in \Sp_G$, and note that $\pi_0^G(\mathbb{S}_G)=A(G)$ the Burnside ring of the finite group $G$. By a result of Dress \cite{Dress}, the ring $A(G)$ is indecomposable if and only if $G$ is solvable. It follows that if $G$ is not solvable then $\mathbb{S}_G$ decomposes non-trivially as a direct product $\mathbb{S}_G[f^{-1}]\times \mathbb{S}_G[(1-f)^{-1}]$ of separable algebras, which we claim are not standard. Indeed suppose that $\mathbb{S}_G[f^{-1}]\simeq \mathbb{D}(X_+)$ for some finite $G$-set $X$. Since $\Phi^e \mathbb{S}_G=\mathbb{S} \in \Sp$ is indecomposable, we must have either $\Phi^e \mathbb{S}_G[f^{-1}]=0$ or $\Phi^e \mathbb{S}_G[(1-f)^{-1}]=0$. In the first case we deduce that $\Phi^e\mathbb{D}(\Sigma_+^\infty X)=0$ so $X=\emptyset$. This is a contradiction as we assumed that $\mathbb{S}_G[f^{-1}]\not =0$. In the second case we must have $\Phi^e \mathbb{S}_G[f^{-1}]=\mathbb{S}_G$. By checking on $e$-geometric fixed points we deduce that $X=\{\ast\}$, so that $\mathbb{S}_G[f^{-1}]\simeq \mathbb{S}_G$. This is a contradiction as $1-f$ is a nontrivial idempotent on $\mathbb{S}_G$ but trivial on $\mathbb{S}_G[f^{-1}]$.
\end{Rem}

However, even if the group is solvable (but not a $p$-group), there may exist non-standard separable algebras. To show this, we need some preliminary results.

\begin{Prop}\label{prop-pullback-pi0}
    Consider a pullback of groupoids
    \[
    \begin{tikzcd}
        \cat A \arrow[r] \arrow[d] & \cat B \arrow[d,"f"]\\
        \cat C \arrow[r,"g"] & \cat D.
    \end{tikzcd}
    \]
    Then the canonical map $\varphi\colon \pi_0(\cat A) \to \pi_0(\cat B) \times_{\pi_0(\cat D)} \pi_0(\cat C)$ is surjective. Moreover given $a=(b,c, \eta \colon f(b) \xrightarrow{\sim} g(c) )\in \pi_0(\cat A)$, there is a bijection 
    \[
    \varphi^{-1}\varphi(a)\xrightarrow{\sim} \doublefaktor{{\mathrm{Aut}_{\cat C}(c)}\op}{\mathrm{Aut}_{\cat D}(f(b))}{\mathrm{Aut}_{\cat B}(b)},
    \]
    where $\mathrm{Aut}_{\cat B}(b)$ acts on the right via $f$, and $\mathrm{Aut}_{\cat C}(c)\op$ acts on the left via the group homomorphism 
     \[
     \mathrm{Aut}_{\cat C}(c)\op \to \mathrm{Aut}_{\cat D}(f(b)), \quad \gamma \mapsto \eta^{-1} g(\gamma)^{-1}\eta.
     \]
\end{Prop}

\begin{proof}
   Pick $[b',c'] \in \pi_0(\cat B) \times_{\pi_0(\cat D)}\pi_0(\cat C)$ so that $[f(b')]=[g(c')]$, that is, there exists an isomorphism $\eta \colon f(b') \xrightarrow{\sim} g(c')$ in $\cat D$. By definition of pullback in groupoids, the triple $(b',c',\eta')$ defines an object $a' \in \cat A$ which satisfies $\varphi(a')=[b',c']$, showing surjectivity of $\varphi$. For the second claim consider $a'=[b',c',\eta'\colon f(b')\xrightarrow{\sim} g(c')]\in\pi_0(\mathcal A)$ and suppose that $\varphi(a)=\varphi(a')$ which means that there exist isomorphisms $\beta \colon b \xrightarrow{\sim} b'$ and $\gamma \colon c \xrightarrow{\sim}c'$. We define a map (of sets)
   \begin{align*}
   \Phi \colon \varphi^{-1}\varphi(a) &\to \doublefaktor{{\mathrm{Aut}_{\cat C}(c)}\op}{\mathrm{Aut}_{\cat D}(f(b))}{\mathrm{Aut}_{\cat B}(b)}\\
   \quad a'=[b',c',\eta']&\mapsto [\eta^{-1}g(\gamma)^{-1}\eta' f(\beta)].
   \end{align*}
   We need to verify that this is well-defined so that the class in the double coset is independent of the chosen $\beta$ and $\gamma$. If we pick other isomorphisms $\beta' \colon b \xrightarrow{\sim} b'$ and $\gamma' \colon c \xrightarrow{\sim} c'$ witnessing that $\varphi(a)=\varphi(a')$, then we can find  $\nu \in \mathrm{Aut}_{\cat B}(b)$ and $\epsilon \in \mathrm{Aut}_{\cat C}(c)$ such that $\beta'=\beta \nu$ and $\gamma'=\gamma \epsilon$. Then the image under $\Phi$ is
   \begin{align*}
    \eta^{-1}g(\gamma')^{-1}\eta' f(\beta') & =\eta^{-1}g(\epsilon)^{-1}g(\gamma)^{-1}\eta'f(\beta)f(\nu)\\
    & =\eta^{-1}g(\epsilon)^{-1}\eta \eta^{-1}g(\gamma)^{-1}\eta'f(\beta)f(\nu)
   \end{align*}
   which defines the same class as $\eta^{-1}g(\gamma)^{-1}\eta' f(\beta)$ in the codomain of $\Phi$.
    There is also a map 
    \[
    \Psi \colon \mathrm{Aut}_{\cat D}(f(b)) \to \varphi^{-1}\varphi(a), \quad \delta \mapsto [b,c, \eta\delta \colon f(b) \xrightarrow{\sim} g(c)].
    \] 
    We first note that $\varphi \Psi(\delta)=[b,c]=\varphi(a)$ so $\Psi$ indeed lands in $\varphi^{-1}\varphi(a)$. To verify that $\Psi$ factors through the double coset we take $\delta\in \mathrm{Aut}_{\cat D}(f(b))$ and $\delta'=\eta^{-1}g(\gamma)^{-1}\eta\delta f(\beta)$ for some $\beta \in \mathrm{Aut}_{\cat B}(b)$ and $\gamma \in \mathrm{Aut}_{\cat C}(c)$. We then need to show that $(b,c,\eta\delta)$ and $(b,c,\eta \delta')=(b,c,g(\gamma)^{-1}\eta\delta f(\beta))$ are isomorphic in $\cat A$. This holds since there is a commutative diagram
    \[
     \begin{tikzcd}
         f(b)\arrow[d,"f(\beta)"',"\sim"] \arrow[rrr,"g(\gamma)^{-1}\eta \delta f(\beta)"] & & & g(c) \arrow[d,"g(\gamma)","\sim"'] \\
         f(b)\arrow[rrr, "\eta\delta"] & & & g(c),
     \end{tikzcd}
    \]
    hence the pair $(\beta,\gamma)$ defines the required isomorphism.
    Finally let us verify that $\Psi$ and $\Phi$ are inverse to one another. To this end, consider $a'=[b',c',\eta'] \in \varphi^{-1}\varphi(a)$ and choose isomorphisms $\beta \colon b \xrightarrow{\sim} b'$ and $\gamma \colon c \xrightarrow{\sim}c'$. Then we find that 
    \[
    \Psi \Phi(a')=[b,c, \eta \eta^{-1}g(\gamma)^{-1}\eta'f(\beta)]=[b,c, g(\gamma)^{-1}\eta' f(\beta)]
    \]
    which is isomorphic in $\cat A$ to $(b',c',\eta')$ as there is a commutative square
    \[
    \begin{tikzcd}
        f(b)\arrow[d,"f(\beta)"',"\sim"] \arrow[rrr,"g(\gamma)^{-1}\eta'f(\beta)"] & & & g(c) \arrow[d,"g(\gamma)","\sim"']\\
        f(b') \arrow[rrr,"\eta'"] &&& g(c')
    \end{tikzcd}
    \]
    showing that $\Psi \Phi(a')=a'$. For the other composite let $\delta\in\mathrm{Aut}_{\cat D}(f(b))$, then 
    \[
    \Phi \Psi(\delta)=\Phi([b,c,\eta \delta])=[\eta^{-1}g(\gamma)^{-1}(\eta \delta)f(\beta)]=[\delta]
    \]
    since we can choose $\beta$ and $\gamma$ to be the corresponding identities. 
\end{proof}

\begin{Lem}\label{lem-iso-unit}
    Let $\CC\in \CAlg(\PrL)$ be such that $\pi_0(\unit)$ is indecomposable. For any $X \in \Fin$, we have 
    \[
    \Map_{\CAlg(\CC)}(\unit^X, \unit^X) \simeq \Map_{\Fin}(X,X).
    \]
    In particular, the automorphism group of the commutative algebra $\unit^X$ is given by the symmetric group $\Sigma_X$.   
\end{Lem}

\begin{proof}
    Note that $\unit^X=\prod_X \unit$. Therefore using \Cref{rem:idempotents}, we calculate
    \begin{align*}
    \Map_{\CAlg(\CC)}(\unit^X, \unit^X) & =\prod_X \coprod_X \Map_{\CAlg(\CC)}(\unit,\unit)\\
    &= \prod_X \coprod_X \Map_{\Fin}(\ast, \ast) \\
    &= \Map_{\Fin}(X,X).\\
    \end{align*}
\end{proof}

\begin{Exa}\label{ex-Cpq}
    Let $G=C_{pq}$ be the cyclic group of order $pq$ for two coprime prime numbers $p$ and $q$. We take the following exhaustive collection of families 
    \[
    \emptyset \subseteq \{e\} \subseteq \{e,C_p\} \subseteq \{e,C_p,C_q\}\subseteq \mathrm{All}.
    \]
     Using \Cref{thm-pullback-stmodfin} with $\mathcal F=\emptyset$ we obtain the following pullback square of groupoids:
    \begin{equation}\label{square-c6}
    \begin{tikzcd}
      \CSep(\Sp_G^\omega)\arrow[r,"L^\omega_{\{e\}}"] \arrow[d,"(\widetilde{\Phi}^e)^\omega"'] & \CSep((\Sp_G /\{e\})^{\omega}) \arrow[d,"\beta_0 \circ (\overline{\Phi}^e)^\omega"]\\
      \CSep(\Fun(BG,\Sp^\omega)) \arrow[r, "\beta_0 \circ q"] & \CSep(\Perf(\mathbb{S}_p^\wedge)^{tG}) \times \CSep( \Perf(\mathbb{S}_q^\wedge )^{tG}).
    \end{tikzcd}
    \end{equation}
    Using \Cref{main-theorem} and the fact that the corresponding Weyl groups are either $p$-groups or $q$-groups, we see that 
    \[
    \CSep((\Sp_G /\{e\})^{\omega}) \simeq(\CFin_G/\{e\})\op,
    \]
    i.e., all separable algebras in $\CSep((\Sp_G /\{e\})^{\omega})$ are standard. Using \Cref{lem-separablyclosed-G}, we also find that 
    \[
    \CSep(\Fun(BG, \Sp^\omega))\simeq \CFin_G\op.
    \]
    We also observe that $(\mathbb{S}_p^\wedge)^{tG}$ and $(\mathbb{S}_q^\wedge)^{tG}$ are indecomposable (argue as in the proof of \Cref{thm-IC}), so the tate categories appearing in the bottom right corner have indecomposable unit.
    By \Cref{prop-pullback-pi0} any object $A \in \CSep(\Sp_G^\omega)$ can be represented by a tuple 
    \[
    A=(A_1,A_2, \eta_1 \colon A_1'\simeq A_2', \eta_2 \colon A_1'' \simeq A_2'' )
    \]
    with
    \begin{itemize}
        \item[(1)] an object $A_1 \in \CSep((\Sp_G /\{e\})^{\omega})$;
        \item[(2)] an object $A_2\in \CSep(\Fun(BG, \Sp^\omega))$;
        \item[(3)] an isomorphism $\eta_1\colon A_1'\simeq A_2'$ between the images of $A_1$ and $A_2$ in $\CSep(\Perf(\mathbb{S}_p^\wedge)^{tG})$;
        \item[(4)] an isomorphism $\eta_2\colon A_1''\simeq A_2''$ between the images of $A_1$ and $A_2$ in $\CSep(\Perf(\mathbb{S}_q^\wedge)^{tG})$.
    \end{itemize}
    We now claim that $A\in\CSep(\Sp_G^\omega)$ determined by $A_1= \unit^{\times 2}$, $A_2=\unit^{\times 2}$, $\eta_1=\mathrm{id}$ and $\eta_2=\sigma \in\Sigma_2$ the flip isomorphism (see \Cref{lem-iso-unit}) is not standard:
    Indeed, since $(\tilde\Phi^e)^\omega$ determines the $G$-set defining a standard algebra, we would have to have $A\simeq \unit^{\times 2}$, which by \Cref{prop-pullback-pi0} would imply
    \[ [(\mathrm{id},\mathrm{id})]=[(\mathrm{id},\sigma)]  \text{ in } \doublefaktor{\Sigma_2}{\Sigma_2 \times \Sigma_2}{\Sigma_2},\]
    which is not true (observe both maps of $\Sigma_2$ are the diagonal embedding).
\end{Exa}
\begin{Rem}
    The argument in the previous example also applies to $\bbZ_G=\infl_e^G \bbZ$. It follows that in the category of compact derived $C_{pq}$-Mackey functor not all separable commutative algebras are standard. 
\end{Rem}
In fact, the above argument is robust enough to give a complete classification of separable algebras in this case. In turn, we will use this to compute the Galois group of $\Sp_{C_{pq}}^\omega$ in the sense of \cite{Mathew2016} for distinct primes $p\neq q$.  Before achieving this we need the following result. 

\begin{Lem}\label{lem-tate-fully-faithful}
    The canonical functors 
    \[
    \Perf(\Sphere_p^{\wedge})^{tC_{pq}}\to (\Perf(\Sphere_p^{\wedge})^{tC_p})^{hC_q}\quad \mathrm{and}\quad \Perf(\Sphere_q^{\wedge})^{tC_{pq}}\to (\Perf(\Sphere_q^{\wedge})^{tC_q})^{hC_p} 
    \]
    are equivalences (recall that the Tate categories are idempotent complete).
\end{Lem}

\begin{proof}
    The situation is symmetric so we only deal with the first functor.
    We observe that there is a solid diagram 
\[
\begin{tikzcd}
    (\Perf(\Sphere^{\wedge}_p)_{hC_p})^{hC_q} \arrow[r, "\mathrm{Nm}_{C_p}^{hC_q}"] & (\Perf(\Sphere_p^{\wedge})^{hC_p})^{hC_q} \arrow[r] & (\Perf(\Sphere_p^{\wedge})^{tC_p})^{hC_q}\\
    \Perf(\Sphere_p^{\wedge})_{hC_{pq}} \arrow[r, "\mathrm{Nm}_{C_{pq}}"] & \Perf(\Sphere_p^{\wedge})^{hC_{pq}}\arrow[r] \arrow[u, "\sim"] & \Perf(\Sphere_p^{\wedge})^{tC_{pq}}
\end{tikzcd}
\]
where the equivalence follows for instance from Fubini's theorem for limits. By definition of the tate categories, the bottom sequence of the diagram is a cofibre sequence (in idempotent stable $\infty$-categories). We claim that the top sequence is also a cofibre sequence. This uses that $q$ is invertible in $\Sphere_p^\wedge$ and so the norm map gives an equivalence $(-)^{hC_q}\simeq (-)_{hC_q}$ and the latter functor preserves cofiber sequences. 
We claim that the image of $(\Perf(\Sphere_p^{\wedge})_{hC_p})^{hC_q}$ and $\Perf(\Sphere_p^{\wedge})_{hC_{pq}}$ under the middle vertical equivalence agree. This would prove the lemma by passing to cofibres.

For one containment, it is enough to note that $\Sphere_p^{\wedge}[C_{pq}]$ is sent to an object of $(\Perf(\Sphere_p^{\wedge})_{hC_p})^{hC_q}$. This follows from the fact that the norm map induces a fully faithful embedding $\Perf(\Sphere_p^{\wedge})_{hC_p}\to \Perf(\Sphere_p^{\wedge})^{hC_p}$ and that, upon taking fixed points, detecting whether an object is in the smaller category can be done underlying, that is, after forgetting the $C_q$-action. For the other containment, we note that since $q$ is invertible, any $X\in(\Perf(\Sphere_p^{\wedge})_{hC_p})^{hC_q}$ is a retract of $X[C_q]$ where the $C_q$-action on $X$ is forgotten. Hence, since the underlying object in $\Perf(\Sphere_p^{\wedge})_{hC_p}$ is in the thick subcategory generated by $\Sphere_p^{\wedge}[C_p]$, we obtain the result. 
\end{proof}

\begin{Thm}\label{thm:sep-in-c6}
    Let $p\neq q$ be prime numbers. We have an equivalence 
    \[
     \CAlg^{\sep}(\Sp_{C_{pq}}^{\omega})\op\simeq \Fin_{C_{pq}} \times_{\Fin} \Fin^{B\Z} 
    \]
    where the pullback is taken along the forgetful functor $\Fin^{B\Z} \to \Fin$ and the $C_{pq}$-fixed points functor $\Fin_{C_{pq}}\to \Fin$.
     In particular, any separable algebra in $\Sp_{C_{pq}}^\omega$ can be written as $A\times B$, where $A$ is standard with proper isotropy, and $B$ corresponds to the $\Fin^{B\mathbb{Z}}$-summand. 
\end{Thm}
\begin{proof}
     Using \Cref{thm-pullback-stmodfin} with $\mathcal F=\emptyset$ we obtain the following pullback square:
    \begin{equation}\label{square-cpq}
    \begin{tikzcd}
      \CAlg^{\sep}(\Sp_{C_{pq}}^\omega)\arrow[r,"L^\omega_{\{e\}}"] \arrow[d,"(\widetilde{\Phi}^e)^\omega"'] & \CAlg^{\sep}((\Sp_{C_{pq}} /\{e\})^{\omega}) \arrow[d,"\beta_0 \circ (\overline{\Phi}^e)^\omega"]\\
      \CAlg^{\sep}(\Fun(BC_{pq},\Sp^\omega)) \arrow[r, "\beta_0 \circ q"] & \CAlg^{\sep}(\Perf(\mathbb{S}_p^\wedge)^{tC_{pq}}) \times \CAlg^{\sep}( \Perf(\mathbb{S}_q^\wedge )^{tC_{pq}}).
    \end{tikzcd}
    \end{equation}
    We claim that 
    \[
    \CAlg^{\sep}((\Sp_{C_{pq}}/\{e\})^\omega)\simeq (\Fin_{C_{pq}}/\{e\})\op \mbox{ and } \CAlg^{\sep}(\Fun(BC_{pq},\Sp^\omega))\simeq \Fin_{C_{pq}}\op.
    \]
    The second equivalence follows from the fact that taking separable algebras commutes with limits. For the first equivalence, we have already argued in \cref{ex-Cpq} that we have an equivalence on groupoids, and so we can conclude by an application of \cref{prop-fully-faulful}.
    
    To prove the theorem we have to identify the maps in the square \eqref{square-cpq} as well as the mapping spaces in the bottom right corner.  Up to the above identification we have a commutative diagram 
    \[
    \begin{tikzcd}
    \Fin_{C_{pq}}\op \arrow[dr, "\Sphere_{C_{pq}}^{(-)}"] \arrow[rrd, bend left, "\mathrm{proj}"]\arrow[ddr, bend right,"\mathrm{id}"']&         & \\
                      & \CAlg^{\sep}(\Sp_{C_{pq}}^{\omega}) \arrow[r,"L^\omega_{\{e\}}"] \arrow[d,"(\widetilde{\Phi}^e)^\omega"'] & (\Fin_{C_{pq}}/\{e\})\op \arrow[d]\\
                      &   \Fin_{C_{pq}}\op \arrow[r] & \CAlg^{\sep}(\Perf(\mathbb{S}_p^\wedge)^{tC_{pq}}) \times \CAlg^{\sep}( \Perf(\mathbb{S}_q^\wedge )^{tC_{pq}}).
    \end{tikzcd}
    \]
    where the square is a pullback, and the tilted functor is \eqref{eq:betterfunctor}. The commutativity of this diagram shows that the bottom horizontal functor identifies with the composite of the projection functor $\Fin_{C_{pq}}\op \to (\Fin_{C_{pq}}/\{e\})\op$ with the right vertical map of the above square. Therefore, it suffices to understand maps in $\CAlg(\Perf(\Sphere_p^\wedge)^{tC_{pq}}\times\Perf(\Sphere_q^\wedge)^{tC_{pq}})$ between standard separable algebras with no trivial isotropy. 

We describe how to do this in $\Perf(\Sphere_p^\wedge)^{tC_{pq}}$, the version with $\Sphere_q^\wedge$ is entirely dual. We first note that $C_{pq}/C_q$ is a retract of $C_{pq}/e$, since $q\in \Sphere_p^\wedge$ is invertible. Therefore, the $C_{pq}$-set $C_{pq}/C_q$ is sent to $0$ in this category, so the functor $\Fin_{C_{pq}}\op\to \CAlg^{\sep}(\Perf(\Sphere_p^\wedge)^{tC_{pq}})$ factors through $\Fin_{C_{pq}}\op/(\{e\},C_q)$. Now recall from \cref{lem-tate-fully-faithful} that there is an equivalence $\Perf(\Sphere_p^\wedge)^{tC_{pq}}\to (\Perf(\Sphere_p^\wedge)^{tC_p})^{hC_q}$.
Since $\Fin_{C_{pq}} \simeq (\Fin_{C_p})^{hC_q}$, the above discussion together with \cref{thm-ic+rc} shows that the map $\Fin_{C_{pq}}\op/(\{e\},C_q)\to \CSep(\Perf(\Sphere_p)^{tC_{pq}})$ is fully faithful. 

Dually, $\CFin_{C_{pq}}\op/(\{e\},C_p)\to \CSep(\Perf(\Sphere_q^{\wedge})^{tC_{pq}})$ is fully faithful. Thus, the pullback we are after is: 
\[\begin{tikzcd}
	P & {(\Fin_{C_{pq}}/\{e\})\op} \\
	{\Fin_{C_{pq}}\op} & {(\Fin_{C_{pq}}/(\{e\},C_q))\op \times (\Fin_{C_{pq}}/(\{e\},C_p))\op}
	\arrow[from=1-1, to=1-2]
	\arrow[from=1-1, to=2-1]
	\arrow[from=1-2, to=2-2]
	\arrow[from=2-1, to=2-2]
\end{tikzcd}\]
where both legs are induced by the projection functors. 

We now aim to describe this pullback, and prove it is equivalent to $\Fin_{C_{pq}}\op\times_{\Fin\op}(\Fin^{B\mathbb Z})\op$, where the left functor is the fixed points functor, and the right functor is the forgetful functor. For simplicity, we do this on opposite categories. 

An object in the pullback in question consists of a tuple $(X,Y,\alpha,\beta)$ where $X,Y$ are $C_{pq}$-sets and $Y$ has no trivial isotropy, and two isomorphisms 

\[
\alpha\colon X/(\{e\}, C_q) \cong Y/C_q \quad 
\mathrm{and} \quad \beta\colon X/(\{e\}, C_p)\cong Y/C_p.
\]
Here, we used the notation from \Cref{cons:modF}. 

To such a tuple, we associate the following object of $\Fin_{C_{pq}}\times_{\Fin}(\Fin^{B\mathbb Z})$: 
$$(X, (\beta^{\mid Y^{C_{pq}}}_{\mid X^{C_{pq}}})^{-1}\circ \alpha^{\mid Y^{C_{pq}}}_{\mid X^{C_{pq}}}) $$
where we note that $\alpha$ and $\beta$ both induce isomorphisms $X^{C_{pq}}\cong Y^{C_{pq}}$. For simplicity of notation, we will typically abusively just write $\alpha,\beta$ also for their co/restrictions. 

This clearly defines a functor on our pullback. Conversely, given $$(Z,f:Z^{C_{pq}}\cong Z^{C_{pq}}),$$ define $(X,Y,\alpha,\beta)$ as follows: $X := Z, Y:= Z/(\{e\}, C_p) \coprod_{Z^{C_{pq}}} Z/(\{e\},C_q)$, where the left map is the canonical inclusion, and the right map is $f^{-1}$ followed by the canonical inclusion.  Note that this is abstractly isomorphic to $Z/\{e\}$, but not canonically.

Furthermore, when removing $C_p$-isotropy (i.e. applying $/C_p$), the inclusion $Z^{C_{pq}}\to Z/(\{e\}, C_q)$ becomes an isomorphism, and so the canonical map $X/(\{e\},C_p)= Z/(\{e\},C_p)\to Y/C_p$ is an isomorphism, which we use to define $\beta$. Similarly for $\alpha$, where we use the other inclusion. 

We now verify that these functors are inverse to one another.

Starting from $(Z,f: Z^{C_{pq}}\cong Z^{C_{pq}})$ and then applying both functors gives $(Z,f)$ back, naturally in $(Z,f)$, as a routine calculation shows. 

For the converse, we start with a tuple $(X,Y,\alpha,\beta)$. Note that the natural maps induce an isomorphism $Y/C_p\coprod_{Y^{C_{pq}}} Y/C_q\cong Y$, so using $\alpha,\beta$, the object $Y$ is naturally isomorphic to $X/(\{e\},C_p)\coprod_{X^{C_{pq}}} X/(\{e\},C_q)$, where the pushout is taken along the natural inclusion on the left, and $\alpha^{-1}\circ \beta$ followed by the natural inclusion on the right. From this observation, one easily concludes that the composite in the other direction is also naturally isomorphic to the identity, which concludes the proof.

Finally, in $\Fin_{C_{pq}}\times_{\Fin}\Fin^{B\mathbb Z}$, we clearly have that any $(Z,f)$ can be written as $(Z\setminus Z^{C_{pq}}, \id_{\emptyset})\coprod (Z^{C_{pq}}, f)$ (though naturally only in isomorphisms). Unwinding the above equivalence, we find that the terms with proper isotropy are in the image of the functor $(\Fin_{C_{pq}})\op\to \CAlg^{\sep}(\Sp_{C_{pq}}^\omega)\simeq P$, which is why every separable algebra can be written $A\times B$ where $A$ is standard with proper isotropy and $B$ corresponding to $\Fin\op\times_{\Fin\op}\Fin^{B\mathbb Z}$ is a ``twist'' of a finite power of the sphere. 
\end{proof}

We can use the above classification of separable algebras to show that the Galois group in this case is non-trivial.

\begin{Cor}\label{cor-galois-cpq}
    For primes $p\neq q$, the Galois group of $\Sp_{C_{pq}}$ is $\widehat{\mathbb Z}$.
\end{Cor}
\begin{proof}
 We have already noted that $\pi_0(\unit)$ is given by the Burnside ring $A(C_{pq})$ which is indecomposable and so the Balmer spectrum of $\Sp_{C_{pq}}^\omega$ is connected by \cite[Proposition 3.7]{NaumannPol}. It follows from \cite[Corollary 8.13]{NaumannPol} that the finite Galois covers are exactly those separable algebras in $\Sp_{C_{pq}}^\omega$ with locally constant (i.e. constant) degree function. In other words, the finite covers are those separable algebras which are descendable. 

In the situation of \Cref{thm:sep-in-c6}, it is clear that if $A\times B$ is a finite cover then so is $A$, but standard separable algebras with proper isotropy are not finite covers (argue as in the proof of \cref{cor-galois}).  Furthermore, by \cite[Lemma 9.6]{NaumannPol}, since any $B\in\Fin^{B\mathbb{Z}}$ maps to a finite product of the unit in either corner of the pullback square \eqref{square-cpq}, we find that they are finite covers. Thus, we find that the the opposite of the category of finite covers of $\Sp_{C_{pq}}^\omega$ is $\Fin^{B\mathbb Z}\simeq \Fin^{\mathrm{cts}}_{\widehat{\mathbb Z}}$, by definition of profinite completion (since each $\Sigma_n$ is finite). 

By definition of the Galois group, cf. \cite[Definition 6.8, Definition 5.39, Theorem 5.36]{Mathew2016}, the result follows. 
\end{proof}

\subsection{Further considerations}
We finish this section with some general comments on the classification of separable commutative algebras in $G$-spectra.
\begin{Rem}
It is a folklore result that after inverting the order of the group, the geometric fixed points functors induces an equivalence
\[
\Sp_G[1/|G|]\xrightarrow{\sim}\prod_{(H)\subseteq G}\Fun(BW_G(H), \Sp[1/|G|]).
\]
Now we know by \cite[Proposition 10.5]{NaumannPol} that the separable commutative algebras in $\Sp[1/|G|]^{\omega}$ are classified by the \'etale fundamental group $\mathrm{Gal}(1/|G|])$ of $\Z[1/|G|]$, which is not zero in general.

Hence, $$\CSep(\Sp[1/|G|]^\omega) \xrightarrow{\sim} \prod_{(H)\subseteq G} \Fun(BW_G(H), \CFin_{\mathrm{Gal}(\Z[1/|G|])})$$
Therefore in this case we have more separable commutative algebras. Since we cannot calculate this \'etale fundamental group, this description is as good as one can hope for.
\end{Rem}
\begin{Rem}
    If $G$ is a $p$-group and we consider separable commutative algebras in $\Sp_{G,(p)}^\omega$, then the classification becomes substantially more intricate than if we invert the group order. In this setting, the $H$-geometric fixed points of a separable commutative algebra are governed by the \'etale fundamental group of $\Z_{(p)}$ (which is highly nontrivial) and its residual $W_G(H)$-action, as above. However, here, the gluing data is also nontrivial which, together with this interplay between the Weyl group and the \'{e}tale fundamental group, makes classification considerably harder.
\end{Rem}
\section{Normed separable algebras}\label{section-normed}
We note that the standard separable algebras in $\Sp_G^\omega$ all come with an important piece of extra structure: multiplicative norms. We refer to \cite{Normpaper} and \cite{NormpaperII} for background on norms in equivariant homotopy theory and normed algebras. 
It is worth wondering whether this extra piece of information is sufficient to completely single them out, i.e. whether our examples of nonstandard separable algebras in $\Sp_G^\omega$ can admit norms. We prove below that our examples over $C_{pq}$ do \emph{not} admit a normed structure. In fact, the situation is somewhat simpler than without norms, in that we do not need to restrict to $p$-groups to get a complete classification, the important feature here is instead solvability: 
\begin{Thm}\label{thm-normed-class}
    Let $G$ be a solvable group and $A\in\CAlg(\Sp_G^\omega)$ a separable commutative algebra. Suppose that $A$ admits norms, that is, there is a normed algebra whose underlying commutative algebra is $A$. In this case, $A$ is standard, and $A$ admits a unique normed structure.  
\end{Thm}
\begin{Rem}
    The restriction to solvable groups is necessary, and is closely related to \Cref{rem-solvable}. The approach we provide here in the solvable case does not seem to extend to a full classification in the non-solvable case, but a full classification seems more approachable than without norms, even in the case of a general finite group $G$. We will explain the complications in the general case after the proof of the above theorem.
\end{Rem}
We will need a couple of lemmas about normed algebras. 
\begin{Lem}\label{lem-nrm-0}
    Let $A\in\CAlg(\Sp_G)$ admitting norms. If $A^e= 0$, then $A=0$. 
\end{Lem}

\begin{proof}
    One of the pieces of data of a normed algebra structure is an algebra map $N^G_e A^e\to A$. By our assumptions,  $N^G_e0 = 0$. The only ring with an algebra map from the zero ring is the zero ring. 
\end{proof}
\begin{Rem}
    The above lemma is false for algebras without norms as witnessed by $\widetilde{E\mathcal P}$ or $\widetilde{EG}$. 
\end{Rem}
\begin{Lem}\label{lm-forgetnorm}
    The forgetful functor $U \colon \mathrm{NAlg}(\Sp_G) \to \CAlg(\Sp_G)$ from normed algebras to commutative algebras preserves (co)limits and it is conservative. In particular, it is both monadic and comonadic.  
\end{Lem}

\begin{proof}
This is the equivariant variant of \cite[Proposition 7.6(3)]{Normpaper}, see also discussion after \cite[Definition 9.14]{Normpaper}. 
\end{proof}
\begin{Lem}\label{lem-monadic}
Let $G$ be a finite group and $H\subseteq G$ be a subgroup. The coinduction functor induces an equivalence
     \[
    \mathrm{NAlg}(\Sp_H) \simeq \mathrm{NAlg}(\Sp_G)_{\Sphere^{G/H}/}.
    \]
\end{Lem}
\begin{proof}
Since $\CoInd_H^G \Sphere\simeq \Sphere^{G/H}$, the coinduction functor induces the first functor in the following composite:
\[
\mathrm{NAlg}(\Sp_H) \to \mathrm{NAlg}(\Sp_G)_{\Sphere^{G/H}/} \xrightarrow{\mathrm{fgt}} \mathrm{NAlg}(\Sp_G).
\]
We note that the forgetful functor is monadic, and that the composite functor is also monadic ($\CoInd_H^G \colon \Sp_H \to \Sp_G$ is monadic by \cite[Theorem 5.32]{MNN}, and the functor which forgets the normed structure preserves enough colimits and it is conservative). Let $T$ be the monad associated to the composite functor, and let $S$ be the monad associated to the second functor. By the universal property, we get a map of monads $S \to T$. If we show that this map is an equivalence, then we prove the lemma since 
\[
\mathrm{NAlg}(\Sp_H)\simeq \Mod_T(\mathrm{NAlg}(\Sp_G)) \simeq \Mod_S(\mathrm{NAlg}(\Sp_G))\simeq \mathrm{NAlg}(\Sp_G)_{\Sphere^{G/H}/}.
\]
In our case, the map of monads is of the form: 
\begin{equation}\label{monad}
\Sphere^{G/H} \otimes A \to  \CoInd^G_H \res_H^G A
\end{equation}
where we used that the monad for $\cat C_{x/} \to \cat C$ is given by the coproduct with $x$. On the $A$ part, it has to be the unit map (by compatibility of the unit maps), and on the $\Sphere^{G/H}$ part, it has to be the special case $A = \Sphere$ (by naturality). So we simply need to check that the map \eqref{monad} induced by these requirements is an equivalence. Since $\mathrm{NAlg}(\Sp_G) \to \CAlg(\Sp_G)$ is conservative, we can check that this map is an equivalence in $ \CAlg(\Sp_G)$, where it holds by the projection formula.
\end{proof}
We are now equipped to prove the theorem, though unlike in the main body of the text we do not require an isotropy separation argument: 
\begin{proof}[Proof of \Cref{thm-normed-class}]
    We proceed by induction on the group $G$ (this is possible since subgroups of solvable groups are solvable). 

    If $G$ is the trivial group, then there is no difference between normed algebras and commutative algebras, so the theorem holds true trivially. This prove the base case.

    Let $A$ be a normed algebra in $\Sp_G^\omega$, and suppose $A$ is separable as a commutative algebra. Since $\Sphere$ is separablly closed (see \Cref{ex-sphere-sepclosed}) we have $A^e\simeq \Sphere^X$ for some finite set $X$ with $G$-action. 

    We note that $N_e^G (\Sphere^X)\simeq \prod_{F(G,X)}\Sphere$ where $F(G,X)$ is the coinduction of $X$. Observe that, as a $G$-set, $F(G,X)\cong X \coprod Y$ for some other finite $G$-set $Y$, where $X$ is embedded as the equivariant functions $G\to X$. It follows that $N_e^G(A^e)\cong \Sphere^X\times \Sphere^Y$ as normed algebras. Base-changing this decomposition along the normed algebra map $N_e^G(A^e)\to A$ and using \Cref{lm-forgetnorm} we obtain a decomposition $A\simeq A_0\times A_1$ as normed algebras.

    Upon taking $e$-fixed points, we find that the map $(N^G_eA^e)^e\to A^e$ is the projection $\prod_{F(G,X)}\Sphere\to \prod_X \Sphere$, so that $A_1^e = 0$. As $A_1$ is normed, \Cref{lem-nrm-0} implies $A_1=0$, so $A=A_0$ has a normed map from $\Sphere^X$, the standard algebra with its normed structure. Decomposing $X$ into orbits and using \cref{lem-monad}
    we reduce, by induction, to the case where $A^e = \Sphere$, where we need to prove that $A\simeq\Sphere_G$. 

    By induction again, we find that for every proper subgroup, $A_{\mid H}\simeq\Sphere_H$. 

    Again, by separable closure of the sphere, we find that $\Phi^G A\simeq \Sphere^Y$ for some finite set $Y$. We now do a case distinction based on the cardinality of $Y$. 
    \begin{enumerate}
        \item  If $|Y|=1$, then we are done since then the map $\Sphere_G\to A$ is an equivalence by isotropy separation. 
        \item If $|Y|=0$, then $A\simeq E\mathcal P_+$ for $\mathcal{P}$ the family of proper subgroups. If $E\mathcal P_+$ has an algebra structure, then the unit map induces a map $\Sphere_G\to E\mathcal P_+\times \widetilde{E\mathcal P}$ which one can check is an equivalence by isotropy separation, so that $\widetilde{E\mathcal P}$ is dualizable. We will see below that this is not possible if $G$ is solvable. 
        \item If $|Y|>1$, we consider instead $A\otimes A$. By \cite[Corollary 5.2]{NaumannPol}, it splits as $A\otimes A\simeq A\times C$ for some commutative algebra $C$. Restricting to proper subgroups and taking geometric fixed points, we see that $C\simeq \widetilde{E\mathcal P}^{Y^2\setminus Y}$. Since $A\otimes A$ is dualizable, the cardinality of $Y$ being $>1$ guarantees that $\widetilde{E\mathcal P}$ is also dualizable, which, as mentioned in the case above, we will preclude in a second. 
    \end{enumerate}
    The above list shows that it suffices to show that for a solvable group, $\widetilde{E\mathcal P}$ is not dualizable. If it is dualizable, then, since it is an idempotent algebra, we must have a splitting $\Sphere_G\simeq \widetilde{E\mathcal P}\times B$ for some other $B$ (see \cite{Maximenote}), and hence we get an idempotent in $\pi_0^G(\Sphere_G)\cong A(G)$, the Burnside ring. As explained in \Cref{rem-solvable}, these do not exist for solvable groups! So we are indeed done.

\end{proof}
Let us now explain what goes wrong in the non-solvable case. Most of the above proof works in that case as well\footnote{Though, of course, it is a proof by induction, so if the overall goal fails, also the individual steps do!}, until the point where we conclude that $\widetilde{E\mathcal P}$ is dualizable. Surprisingly, in the non-solvable case, this can in fact happen! 
\begin{Exa}\label{ex-EP-finite}
In \cite{Kremer}, Kremer recalls a construction of Floyd--Richardson showing that for $G=A_5$, the universal space $E\mathcal P$ admits a finite $A_5$-CW-model.  He also recalls a result of Bob Oliver \cite{oliver} which implies in particular that the same thing happens for all $G$ which are minimally non-solvable, i.e. non-solvable but every proper subgroup is. 
\end{Exa}

In fact, when this happens, it always leads to nonstandard, compact, normed, separable algebras.  
\begin{Prop}\label{prop-EP-normed}
    Let $G$ be a finite group, and suppose $E\mathcal{P}_+$ is dualizable in $\Sp_G$. In this case, it admits a unique normed structure, and furthermore the algebra structure is separable.
\end{Prop}
There are various ways of arguing, and we offer one which includes interesting intermediate statements. 

In particular, we use the following result, which appears to be new: 
\begin{Prop}\label{prop-norm-res}
    Let $A\in \mathrm{NAlg}(\Sp_G)$ be a normed algebra and consider the full subcategory of the slice category 
    \[
    \mathrm{NAlg}(\Sp_G)^{\mathrm{pr}}_{A/} \subseteq \mathrm{NAlg}(\Sp_G)_{A/}
    \] 
    spanned by those normed algebras $B$ under $A$ such that $A\to B$ restricts to an equivalence on all proper subgroups. The forgetful functor 
    \[
      \mathrm{NAlg}(\Sp_G)^{\mathrm{pr}}_{A/} \to \CAlg(\Sp_G)_{A/}
    \]
    is fully faithful, with essential image given by the commutative algebras $B$ under $A$ such that $A\to B$ restricts to an equivalence on all proper subgroups. 
\end{Prop}
More informally, the above result says that if $A\to B$ is a map of commutative algebras that restricts to an equivalence on all proper subgroups, then any normed algebra structure on $A$ extends to a unique normed algebra structure on $B$. 

To establish this, we will use:
\begin{Lem}\label{lem-monad}
    Let $N\colon  \CAlg(\Sp_G)\to \CAlg(\Sp_G)$ be the monad corresponding to the forgetful functor from normed algebras to commutative algebras. Let $f \colon A\to B$ be a map of commutative algebras, and suppose it is an equivalence upon restriction to every proper subgroup. In this case, the following naturality square is a pushout of commutative algebras:
\[\begin{tikzcd}
	A & B \\
	NA & NB.
	\arrow["f",from=1-1, to=1-2]
	\arrow[from=1-1, to=2-1]
	\arrow[from=1-2, to=2-2]
	\arrow["N(f)",from=2-1, to=2-2]
\end{tikzcd}\]
\end{Lem}
\begin{proof}
The bar resolution shows that any such $f\colon A\to B$ is a (sifted) colimit of maps between free commutative algebras on objects of $\Sp_G$, all satisfying the same restriction property. Thus, without loss of generality we may assume $A,B$ are free, say on $X,Y\in\Sp_G$, and the map $f$ is induced by a map $X\to Y$ whose restriction to proper subgroups is an equivalence. 

We need to show that the above square is a pushout in this case. Since geometric fixed points and restriction to proper subgroups are symmetric monoidal, colimit preserving and jointly conservative, this can be checked after either operation. After restricting to proper subgroups, both the top and bottom horizontal maps become equivalences, so the claim is clear, and we reduce to checking it on geometric fixed points. 

There, we claim that if $A$ is free on $X$, $\Phi^G(NA)$ splits as $\Phi^G(A)\otimes F(X_{\mid \mathcal P}) $ for some functor $F$ depending only on the restriction of $X$ to proper subgroups. 

This follows from the formula  $NA\simeq \underline{\colim}_{G/H\in\underline{\Fin}^\simeq} N_H^G X$ from \cite[Corollary 3.48]{NormpaperII}. Here, $\underline{\Fin}^\simeq$ is the $G$-groupoid of finite sets, whose $H$-fixed points are $\CFin_H$ for every $H$, and $\underline{\colim}$ means we are taking a $G$-colimit (or parametrized colimit) over a certain $G$-functor $\underline{\Fin}^\simeq \to \underline{\Sp}$, informally described as sending a finite $H$-set to the corresponding norm of $X$. 

In any case, the geometric fixed points of this $G$-colimit are given by the ordinary colimit of the geometric fixed points indexed by the ordinary category $\CFin_G$ (see \cite[Observation 4.2.2]{HKK2024}), which yields $$\Phi^G NA \simeq \colim_{S\in \CFin_G}\Phi^GN^S (X_{\mid S}).$$ Using that $\CFin_G \simeq \CFin \times \CFin_\mathcal P$ we can rewrite this colimit as $\colim_{S\in \CFin}\Phi^GX^{\otimes S}\otimes \colim_{S\in \CFin_{\mathcal P}} \Phi^GN^S(X_{\mid S})$, and since $A$ is free on $X$, the first tensor summand is exactly $\Phi^G A$, while the second tensor summand only depends on $X_{\mid \mathcal P}$, as was to be shown. 

Having this splitting, naturally in $X$, the pushout claim follows since tensor products are coproducts of commutative algebras. 
\end{proof}
\begin{proof}[Proof of \Cref{prop-norm-res}]
    By a slight abuse of notation we denote by $N$ also the left adjoint to the forgetful functor $U \colon \CAlg(\Sp_G)\to \mathrm{NAlg}(\Sp_G)$. Passing to the slice under $A$, we still have a forgetful functor $U: \mathrm{NAlg}(\Sp_G)_{A/}\to \mathrm{CAlg}(\Sp_G)_{A/}$, but now its left adjoint is given by $B\mapsto A\otimes_{NA}NB$, because relative tensor products are pushouts. The unit of the adjunction is induced by the unit $B\to NB$, and the counit by the naturality square 
    \[\begin{tikzcd}
	NA & NB \\
	A & B.
	\arrow["N(f)",from=1-1, to=1-2]
	\arrow[from=1-1, to=2-1]
	\arrow[from=1-2, to=2-2]
	\arrow["f",from=2-1, to=2-2]
\end{tikzcd}\]
Now when $B$ is a  commutative algebra with a map from $A$ whose restriction to proper subgroups is an equivalence, \Cref{lem-monad} together with pushout pasting shows that the map to the pushout $B\to A\otimes_{NA}NB$ is an equivalence, i.e. the unit of the adjunction restricts to an equivalence on such $B$'s. Since the right adjoint is conservative (see \cref{lm-forgetnorm}), it follows that the whole restricted adjunction is an equivalence on this subcategory, as was to be shown. 
\end{proof}

\begin{proof}[Proof of \Cref{prop-EP-normed}]
If $E\mathcal{P}_+$ is dualizable, then since it is a coidempotent coalgebra, we obtain a splitting $\Sphere_G\simeq E\mathcal{P}_+\times \widetilde{E\mathcal P}$ which witnesses $E\mathcal{P}_+$ as an idempotent algebra, see \cite{Maximenote}. Furthermore, the unit map $\Sphere_G\to E\mathcal{P}_+$ then induces an equivalence on every proper subgroup, so that by \Cref{prop-norm-res}, the algebra $E\mathcal{P}_+$ admits a unique normed structure. 
\end{proof}

In fact, this is not the only counterexample: when $G$ is minimally non-solvable (that is, $G$ is not solvable but all its proper subgroups are; e.g. $G=A_5$), then it follows from \Cref{ex-EP-finite} that $E\mathcal P_+$ is dualizable and hence the above applies, but the same argument shows that $\Sphere_G\times \widetilde{E\mathcal P}^X$ is also uniquely normed for every finite set $X$. Overall, in the induction ``strategy'' from the proof of \Cref{thm-normed-class}, we see that if $G$ is non-solvable, then every non-solvable subgroup of $G$ may contribute to nonstandard separable normed algebras. 

A general classification therefore seems combinatorially involved, but it appears also to be more tractable than the case of not-necessarily-normed separable algebras.  
\section{Other examples}
In this final section we record some other classification results. 

\subsection{Free $G$-spectra}
Consider $R \in\CAlg(\Sp_G)$ such that $R \otimes  \widetilde{E}G\simeq 0$ and $\Phi^e(R)$ separably closed. An example to keep in mind is Atiyah's K-theory with reality where $G=C_2$, see \cite[Example 8.15]{BCHNP}; note that $\Phi^e KR=KU$ is separably closed by \cite[Theorem 12.9]{NaumannPol}.
Consider the pullback square from \Cref{cor-pullback-perfect-G-spectra} with $R$ as above, $\mathcal{F}=\emptyset$ and $K=e$. The right top vertex of the square is 
\[
(\Mod_{\Sp_G}(R)/\{e\})^\omega \simeq \Mod_{\Sp_G}(R \otimes \widetilde{E}G)^\omega\simeq 0
\]
It follows that the functor $\Phi^e \colon \Mod_{\Sp_G}(R)^\omega \to \Perf(\Phi^e R)^{hG}$ is fully faithful. 
Passing to separable algebras we obtain a fully faithful functor 
\[
\CAlg^{\sep}(\Mod_{\Sp_G}(R)^\omega) \hookrightarrow  \CAlg^{\sep}(\Perf(\Phi^e( R)))^{hG} \simeq \Fin_{G}\op
\]
by our assumption on $\Phi^e(R)$. But $R \otimes \mathbb{D}(\Sigma_+^\infty X)\in \CAlg(\Mod_{\Sp_G}(R)^\omega)$ for $X\in \Fin_G$ are always separable, so the above map is an equivalence.

\subsection{Borel $G$-spectra}
Consider $R \in \CAlg$ and its associated Borel $G$-spectrum $\underline{R}\in \CAlg(\Sp_G)$. 
Then it follows from~\cite[Corollary 6.21]{MNN} that there is a fully faithful functor 
\[
\Mod_{\Sp_G}(\underline{R})^\omega \hookrightarrow \Fun(BG, \Mod(R))^\omega \subseteq \Fun(BG, \Perf(R)).
\]
given by $\Phi^e$. Passing to separable algebras we find that 
\[
\CAlg^{\sep}(\Mod_{\Sp_G}(\underline{R})^\omega) \hookrightarrow  \Fun(BG, \CAlg^{\sep}(\Perf(R))).
\]
If $R$ is separably closed, so that $\CAlg^{\sep}(\Perf(R)) \simeq \Fin\op$, then the above functor is also essentially surjective since $\underline{R} \otimes \mathbb{D}(\Sigma_+^\infty X)$ is a separable commutative algebra in $\Mod_{\Sp_G}(\underline{R})^\omega$ hitting the finite $G$-set $X$.

\bibliographystyle{alpha}
\bibliography{reference}

@article {NaumannPol,
    AUTHOR = {Naumann, Niko and Pol, Luca},
     TITLE = {Separable commutative algebras and {G}alois theory in stable
              homotopy theories},
   JOURNAL = {Adv. Math.},
  FJOURNAL = {Advances in Mathematics},
    VOLUME = {449},
      YEAR = {2024},
     PAGES = {Paper No. 109736},
      ISSN = {0001-8708,1090-2082},
   MRCLASS = {14A30 (13B05 14F20 18G80 20C20 55U35)},
  MRNUMBER = {4752740},
       DOI = {10.1016/j.aim.2024.109736},
       URL = {https://doi.org/10.1016/j.aim.2024.109736},
}

@unpublished{Kremer,
    author = {Kremer, Christian},
    title = {A note on {P}oincaré duality and ambidexterity for universal spaces for families},
    note = {Available at author's webpage \href{link}{https://drive.google.com/file/d/1UJBxqYHjgchvavaDqSBFwG8-0yRX666i/view}}
}

@unpublished{Maximenote,
    author = {Ramzi, Maxime},
    title = {Small idempotent algebras},
    note = {Available at author's webpage \href{link}{https://sites.google.com/view/maxime-ramzi-en/notes/small-idempotent-algebras}}
}

@article{oliver,
  title={Smooth compact Lie group actions on disks},
  author={Oliver, Robert},
  journal={Mathematische Zeitschrift},
  volume={149},
  number={1},
  pages={79--96},
  year={1976},
  publisher={Springer-Verlag Berlin/Heidelberg}
}

@ARTICLE{Ramzi2023,
       author = {{Ramzi}, Maxime},
        title = "{Separability in homotopical algebra}",
      journal = {arXiv e-prints},
         year = 2023,
        month = may,
          eid = {arXiv:2305.17236},
        pages = {arXiv:2305.17236},
          doi = {10.48550/arXiv.2305.17236},
archivePrefix = {arXiv},
       eprint = {2305.17236},
 primaryClass = {math.AT},
       adsurl = {https://ui.adsabs.harvard.edu/abs/2023arXiv230517236R}
}

@article {Balmer2011,
    AUTHOR = {Balmer, Paul},
     TITLE = {Separability and triangulated categories},
   JOURNAL = {Adv. Math.},
  FJOURNAL = {Advances in Mathematics},
    VOLUME = {226},
      YEAR = {2011},
    NUMBER = {5},
     PAGES = {4352--4372},
      ISSN = {0001-8708},
   MRCLASS = {18E30 (13B40 16H05 18C20)},
  MRNUMBER = {2770453},
MRREVIEWER = {Matthias K\"{u}nzer},
       DOI = {10.1016/j.aim.2010.12.003},
       URL = {https://doi.org/10.1016/j.aim.2010.12.003},
}

@article {Balmer2015,
    AUTHOR = {Balmer, Paul},
     TITLE = {Stacks of group representations},
   JOURNAL = {J. Eur. Math. Soc. (JEMS)},
  FJOURNAL = {Journal of the European Mathematical Society (JEMS)},
    VOLUME = {17},
      YEAR = {2015},
    NUMBER = {1},
     PAGES = {189--228},
      ISSN = {1435-9855},
   MRCLASS = {20C20 (18F10 18F20)},
  MRNUMBER = {3312406},
MRREVIEWER = {Nadia P. Mazza},
       DOI = {10.4171/JEMS/501},
       URL = {https://doi.org/10.4171/JEMS/501},
}

@book{HA,
   author = {Lurie, Jacob},
   title = "{Higher algebra}",
   year={2017},
   note = {Avaliable from the author's webpage at \url{http://www.math.harvard.edu/~lurie/papers/HA.pdf}}
}

@book {HTT,
    AUTHOR = {Lurie, Jacob},
     TITLE = {Higher topos theory},
    SERIES = {Annals of Mathematics Studies},
    VOLUME = {170},
 PUBLISHER = {Princeton University Press, Princeton, NJ},
      YEAR = {2009},
     PAGES = {xviii+925},
      ISBN = {978-0-691-14049-0; 0-691-14049-9},
   MRCLASS = {18-02 (18B25 18E35 18G30 18G55 55U40)},
  MRNUMBER = {2522659},
MRREVIEWER = {Mark Hovey},
       DOI = {10.1515/9781400830558},
       URL = {https://doi.org/10.1515/9781400830558},
}

@article {Mathew2016,
    AUTHOR = {Mathew, Akhil},
     TITLE = {The {G}alois group of a stable homotopy theory},
   JOURNAL = {Adv. Math.},
  FJOURNAL = {Advances in Mathematics},
    VOLUME = {291},
      YEAR = {2016},
     PAGES = {403--541},
      ISSN = {0001-8708},
   MRCLASS = {14F35 (11F11 11S20 14F20 18D10 18G55 55P43 55U35)},
  MRNUMBER = {3459022},
MRREVIEWER = {Rui Miguel Saramago},
       DOI = {10.1016/j.aim.2015.12.017},
       URL = {https://doi.org/10.1016/j.aim.2015.12.017},
}

@ARTICLE{Mathew2015,
       author = {{Mathew}, Akhil},
        title = "{Torus actions on stable module categories, Picard groups, and localizing subcategories}",
      journal = {arXiv e-prints},
         year = 2015,
        month = dec,
          eid = {arXiv:1512.01716},
        pages = {arXiv:1512.01716},
          doi = {10.48550/arXiv.1512.01716},
archivePrefix = {arXiv},
       eprint = {1512.01716},
 primaryClass = {math.RT},
       adsurl = {https://ui.adsabs.harvard.edu/abs/2015arXiv151201716M}
}

@article {MNN,
    AUTHOR = {Mathew, Akhil and Naumann, Niko and Noel, Justin},
     TITLE = {Nilpotence and descent in equivariant stable homotopy theory},
   JOURNAL = {Adv. Math.},
  FJOURNAL = {Advances in Mathematics},
    VOLUME = {305},
      YEAR = {2017},
     PAGES = {994--1084},
      ISSN = {0001-8708},
   MRCLASS = {55P91 (55P42)},
  MRNUMBER = {3570153},
MRREVIEWER = {Gregory Z. Arone},
       DOI = {10.1016/j.aim.2016.09.027},
       URL = {https://doi.org/10.1016/j.aim.2016.09.027},
}

@article {Rognes2008,
    AUTHOR = {Rognes, John},
     TITLE = {Galois extensions of structured ring spectra. {S}tably
              dualizable groups},
   JOURNAL = {Mem. Amer. Math. Soc.},
  FJOURNAL = {Memoirs of the American Mathematical Society},
    VOLUME = {192},
      YEAR = {2008},
    NUMBER = {898},
     PAGES = {viii+137},
      ISSN = {0065-9266},
   MRCLASS = {55P43 (55M05 55P35 57T05)},
  MRNUMBER = {2387923},
MRREVIEWER = {Alberto Cavicchioli},
       DOI = {10.1090/memo/0898},
       URL = {https://doi.org/10.1090/memo/0898},
}

@article {Balmer2016,
    AUTHOR = {Balmer, Paul},
     TITLE = {Separable extensions in tensor-triangular geometry and
              generalized {Q}uillen stratification},
   JOURNAL = {Ann. Sci. \'{E}c. Norm. Sup\'{e}r. (4)},
  FJOURNAL = {Annales Scientifiques de l'\'{E}cole Normale Sup\'{e}rieure. Quatri\`eme
              S\'{e}rie},
    VOLUME = {49},
      YEAR = {2016},
    NUMBER = {4},
     PAGES = {907--925},
      ISSN = {0012-9593},
   MRCLASS = {18E30 (13B22 20C05)},
  MRNUMBER = {3552016},
MRREVIEWER = {Srikanth B. Iyengar},
       DOI = {10.24033/asens.2298},
       URL = {https://doi.org/10.24033/asens.2298},
}

@article {BS2017,
    AUTHOR = {Balmer, Paul and Sanders, Beren},
     TITLE = {The spectrum of the equivariant stable homotopy category of a
              finite group},
   JOURNAL = {Invent. Math.},
  FJOURNAL = {Inventiones Mathematicae},
    VOLUME = {208},
      YEAR = {2017},
    NUMBER = {1},
     PAGES = {283--326},
      ISSN = {0020-9910},
   MRCLASS = {18E30 (55P42 55U35)},
  MRNUMBER = {3621837},
MRREVIEWER = {Geoffrey M. L. Powell},
       DOI = {10.1007/s00222-016-0691-3},
       URL = {https://doi.org/10.1007/s00222-016-0691-3},
}

@article {Neeman2018,
    AUTHOR = {Neeman, Amnon},
     TITLE = {Separable monoids in {${\bf D}_{\bf qc}(X)$}},
   JOURNAL = {J. Reine Angew. Math.},
  FJOURNAL = {Journal f\"{u}r die Reine und Angewandte Mathematik. [Crelle's
              Journal]},
    VOLUME = {738},
      YEAR = {2018},
     PAGES = {237--280},
      ISSN = {0075-4102},
   MRCLASS = {18E30 (14F05)},
  MRNUMBER = {3794893},
MRREVIEWER = {M. M. Al-Shomrani},
       DOI = {10.1515/crelle-2015-0039},
       URL = {https://doi.org/10.1515/crelle-2015-0039},
}

@article {Normpaper,
    AUTHOR = {Bachmann, Tom and Hoyois, Marc},
     TITLE = {Norms in motivic homotopy theory},
   JOURNAL = {Ast\'{e}risque},
  FJOURNAL = {Ast\'{e}risque},
    NUMBER = {425},
      YEAR = {2021},
     PAGES = {207},
      ISSN = {0303-1179},
      ISBN = {978-2-85629-939-5},
   MRCLASS = {14F42 (19E15)},
  MRNUMBER = {4288071},
       DOI = {10.24033/ast},
       URL = {https://doi.org/10.24033/ast},
}

@article{NormpaperII,
  title={Norms in equivariant homotopy theory},
  author={Lenz, Tobias and Linskens, Sil and P{\"u}tzst{\"u}ck, Phil},
  journal={arXiv preprint arXiv:2503.02839},
  year={2025}
}

@article {BC2018,
    AUTHOR = {Balmer, Paul and Carlson, Jon F.},
     TITLE = {Separable commutative rings in the stable module category of
              cyclic groups},
   JOURNAL = {Algebr. Represent. Theory},
  FJOURNAL = {Algebras and Representation Theory},
    VOLUME = {21},
      YEAR = {2018},
    NUMBER = {2},
     PAGES = {399--417},
      ISSN = {1386-923X},
   MRCLASS = {20C20 (14F20 18E30)},
  MRNUMBER = {3780773},
MRREVIEWER = {Caroline Lassueur},
       DOI = {10.1007/s10468-017-9719-7},
       URL = {https://doi.org/10.1007/s10468-017-9719-7},
}

@article {BCHNP,
    AUTHOR = {Barthel, Tobias and Castellana, Nat\`alia and Heard, Drew and
              Naumann, Niko and Pol, Luca},
     TITLE = {Quillen stratification in equivariant homotopy theory},
   JOURNAL = {Invent. Math.},
  FJOURNAL = {Inventiones Mathematicae},
    VOLUME = {239},
      YEAR = {2025},
    NUMBER = {1},
     PAGES = {219--285},
      ISSN = {0020-9910,1432-1297},
   MRCLASS = {55P91 (18F99 55P42 55U35)},
  MRNUMBER = {4841779},
       DOI = {10.1007/s00222-024-01301-0},
       URL = {https://doi.org/10.1007/s00222-024-01301-0},
}

@article {Balmer2016-etale,
    AUTHOR = {Balmer, Paul},
     TITLE = {The derived category of an \'{e}tale extension and the separable
              {N}eeman-{T}homason theorem},
   JOURNAL = {J. Inst. Math. Jussieu},
  FJOURNAL = {Journal of the Institute of Mathematics of Jussieu. JIMJ.
              Journal de l'Institut de Math\'{e}matiques de Jussieu},
    VOLUME = {15},
      YEAR = {2016},
    NUMBER = {3},
     PAGES = {613--623},
      ISSN = {1474-7480},
   MRCLASS = {14F20 (18E30)},
  MRNUMBER = {3505660},
MRREVIEWER = {Javier Majadas},
       DOI = {10.1017/S1474748014000449},
       URL = {https://doi.org/10.1017/S1474748014000449},
}

@ARTICLE{Glueingpaper,
       author = {{Naumann}, Niko and {Pol}, Luca and {Ramzi}, Maxime},
        title = "{A symmetric monoidal fracture square}",
      journal = {arXiv e-prints},
         year = 2024,
        month = nov,
          eid = {arXiv:2411.05467},
        pages = {arXiv:2411.05467},
          doi = {10.48550/arXiv.2411.05467},
archivePrefix = {arXiv},
       eprint = {2411.05467},
 primaryClass = {math.AT},
       adsurl = {https://ui.adsabs.harvard.edu/abs/2024arXiv241105467N}
}

@ARTICLE{Krause20,
       author = {{Krause}, Achim},
        title = "{The Picard group in equivariant homotopy theory via stable module categories}",
      journal = {arXiv e-prints},
         year = 2020,
        month = aug,
          eid = {arXiv:2008.05551},
        pages = {arXiv:2008.05551},
          doi = {10.48550/arXiv.2008.05551},
archivePrefix = {arXiv},
       eprint = {2008.05551},
 primaryClass = {math.AT},
       adsurl = {https://ui.adsabs.harvard.edu/abs/2020arXiv200805551K}
}

@article {Nikolaus-Scholze,
    AUTHOR = {Nikolaus, Thomas and Scholze, Peter},
     TITLE = {On topological cyclic homology},
   JOURNAL = {Acta Math.},
  FJOURNAL = {Acta Mathematica},
    VOLUME = {221},
      YEAR = {2018},
    NUMBER = {2},
     PAGES = {203--409},
      ISSN = {0001-5962,1871-2509},
   MRCLASS = {55U35 (16E40 18E30 19D99)},
  MRNUMBER = {3904731},
MRREVIEWER = {Geoffrey\ M. L. Powell},
       DOI = {10.4310/ACTA.2018.v221.n2.a1},
       URL = {https://doi.org/10.4310/ACTA.2018.v221.n2.a1},
}

@article {CMNN2020,
    AUTHOR = {Clausen, Dustin and Mathew, Akhil and Naumann, Niko and Noel,
              Justin},
     TITLE = {Descent and vanishing in chromatic algebraic {$K$}-theory via
              group actions},
   JOURNAL = {Ann. Sci. \'Ec. Norm. Sup\'er. (4)},
  FJOURNAL = {Annales Scientifiques de l'\'Ecole Normale Sup\'erieure.
              Quatri\`eme S\'erie},
    VOLUME = {57},
      YEAR = {2024},
    NUMBER = {4},
     PAGES = {1135--1190},
      ISSN = {0012-9593,1873-2151},
   MRCLASS = {18N60 (19D10 19L47)},
  MRNUMBER = {4773302},
}

@article {Balmer_Ambrogio_Sanders,
    AUTHOR = {Balmer, Paul and Dell'Ambrogio, Ivo and Sanders, Beren},
     TITLE = {Restriction to finite-index subgroups as \'{e}tale extensions
              in topology, {KK}-theory and geometry},
   JOURNAL = {Algebr. Geom. Topol.},
  FJOURNAL = {Algebraic \& Geometric Topology},
    VOLUME = {15},
      YEAR = {2015},
    NUMBER = {5},
     PAGES = {3025--3047},
      ISSN = {1472-2747,1472-2739},
   MRCLASS = {55P91 (13D09 14F05 19K35 19L47)},
  MRNUMBER = {3426702},
MRREVIEWER = {Andr\'{e}\ G.\ Henriques},
       DOI = {10.2140/agt.2015.15.3025},
       URL = {https://doi.org/10.2140/agt.2015.15.3025},
}

@article {Dress,
    AUTHOR = {Dress, Andreas},
     TITLE = {A characterisation of solvable groups},
   JOURNAL = {Math. Z.},
  FJOURNAL = {Mathematische Zeitschrift},
    VOLUME = {110},
      YEAR = {1969},
     PAGES = {213--217},
      ISSN = {0025-5874,1432-1823},
   MRCLASS = {20.80},
  MRNUMBER = {248239},
MRREVIEWER = {R.\ Schmidt},
       DOI = {10.1007/BF01110213},
       URL = {https://doi.org/10.1007/BF01110213},
}

@article {PSW,
    AUTHOR = {Patchkoria, Irakli and Sanders, Beren and Wimmer, Christian},
     TITLE = {The spectrum of derived {M}ackey functors},
   JOURNAL = {Trans. Amer. Math. Soc.},
  FJOURNAL = {Transactions of the American Mathematical Society},
    VOLUME = {375},
      YEAR = {2022},
    NUMBER = {6},
     PAGES = {4057--4105},
      ISSN = {0002-9947,1088-6850},
   MRCLASS = {18G80 (19A99 55P91 55U35)},
  MRNUMBER = {4419053},
MRREVIEWER = {J.\ D.\ Quigley},
       DOI = {10.1090/tran/8485},
       URL = {https://doi.org/10.1090/tran/8485},
}

@article {CLS2023,
    AUTHOR = {Cnossen, Bastiaan and Lenz, Tobias and Linskens, Sil},
     TITLE = {Partial parametrized presentability and the universal property
              of equivariant spectra},
   JOURNAL = {Trans. Amer. Math. Soc.},
  FJOURNAL = {Transactions of the American Mathematical Society},
    VOLUME = {378},
      YEAR = {2025},
    NUMBER = {10},
     PAGES = {6913--6974},
      ISSN = {0002-9947,1088-6850},
   MRCLASS = {55U35 (55P91)},
  MRNUMBER = {4956364},
       DOI = {10.1090/tran/9497},
       URL = {https://doi.org/10.1090/tran/9497},
}

@article{Rickard,
title = {Derived categories and stable equivalence},
journal = {Journal of Pure and Applied Algebra},
volume = {61},
number = {3},
pages = {303-317},
year = {1989},
issn = {0022-4049},
doi = {https://doi.org/10.1016/0022-4049(89)90081-9},
url = {https://www.sciencedirect.com/science/article/pii/0022404989900819},
author = {Jeremy Rickard}
}

@InProceedings{descent,
author="Barthel, Tobias
and Castellana, Nat{\`a}lia
and Heard, Drew
and Naumann, Niko
and Pol, Luca
and Sanders, Beren",
title="Descent in Tensor Triangular Geometry",
booktitle="Triangulated Categories in Representation Theory and Beyond",
year="2024",
publisher="Springer Nature Switzerland",
address="Cham",
pages="1--56",
isbn="978-3-031-57789-5"
}

@misc{stacks-project,
  author       = {The {Stacks project authors}},
  title        = {The Stacks project},
  howpublished = {\url{https://stacks.math.columbia.edu}},
  year         = {2024},
}

@ARTICLE{HKK2024,
       author = {{Hilman}, Kaif and {Kirstein}, Dominik and {Kremer}, Christian},
        title = "{Parametrised Poincar{\'e} duality and equivariant fixed points methods}",
      journal = {arXiv e-prints},
         year = 2024,
        month = may,
          eid = {arXiv:2405.17641},
        pages = {arXiv:2405.17641},
          doi = {10.48550/arXiv.2405.17641},
archivePrefix = {arXiv},
       eprint = {2405.17641},
 primaryClass = {math.AT},
       adsurl = {https://ui.adsabs.harvard.edu/abs/2024arXiv240517641H}
}

\end{document}